\documentclass[a4,12pt]{amsart}
\usepackage{amssymb,amstext,amsmath,amscd,amsthm,amsfonts,enumerate,latexsym}
\usepackage{color}
\usepackage[dvipdfmx]{graphicx}
\usepackage[all]{xy}

\usepackage[dvipdfmx]{hyperref}
\usepackage{ytableau}

\theoremstyle{plain}
\newtheorem{thm}{Theorem}[section]

\newtheorem*{thm*}{Theorem}
\newtheorem*{cor*}{Corollary}

\newtheorem{prop}[thm]{Proposition}

\newtheorem{lem}[thm]{Lemma}
\newtheorem{cor}[thm]{Corollary}

\newtheorem*{claim*}{Claim}

\theoremstyle{definition}
\newtheorem{defn}[thm]{Definition}

\newtheorem{ex}[thm]{Example}

\newtheorem{remark}[thm]{Remark}

\newtheorem*{case*}{Case}
\newtheorem{setting}[thm]{Setting}

\newtheorem{prob}[thm]{Problem}

\theoremstyle{remark}

\numberwithin{equation}{thm}

\newtheorem*{ac}{Acknowledgments}

\def\GCD{\operatorname{GCD}}
\def\Ext{\operatorname{Ext}}

\def\Im{\operatorname{Im}}

\def\Ker{\operatorname{Ker}}

\def\bbZ{\mathbb{Z}}

\def\Hom{\operatorname{Hom}}

\def\mod{\mathrm{mod}}

\def\m{\mathfrak m}
\def\n{\mathfrak n}

\newcommand{\rma}{\mathrm{a}}

\newcommand{\rmc}{\mathrm{c}}

\newcommand{\rme}{\mathrm{e}}
\newcommand{\rmf}{\mathrm{f}}

\newcommand{\rmo}{\mathrm{o}}

\newcommand{\rmr}{\mathrm{r}}

\newcommand{\rmH}{\mathrm{H}}
\newcommand{\rmI}{\mathrm{I}}

\newcommand{\calX}{\mathcal{X}}

\newcommand{\fkc}{\mathfrak{c}}

\newcommand{\mapright}[1]{%
\smash{\mathop{%
\hbox to 1cm{\rightarrowfill}}\limits^{#1}}}

\newcommand{\mapleft}[1]{%
\smash{\mathop{%
\hbox to 1cm{\leftarrowfill}}\limits_{#1}}}

\def\AGL{\operatorname{AGL}}

\def\ch{\operatorname{ch}}

\def\gr{\mbox{\rm gr}}

\title[Ulrich ideals in numerical semigroup rings of small multiplicity]
{Ulrich ideals in numerical semigroup rings \\ of small multiplicity}

\author[Naoki Endo]{Naoki Endo}
\address{Department of Mathematics, Faculty of Science, Tokyo University of Science, 1-3 Kagurazaka, Shinjuku, Tokyo 162-8601, Japan}
\email{nendo@rs.tus.ac.jp}
\urladdr{https://www.rs.tus.ac.jp/nendo/}

\author[Shiro Goto]{Shiro Goto}
\address{Department of Mathematics, School of Science and Technology, Meiji University, 1-1-1 Higashi-mita, Tama-ku, Kawasaki 214-8571, Japan}
\email{shirogoto@gmail.com}

\thanks{The first author was partially supported by JSPS Grant-in-Aid for Young Scientists 20K14299. The second author was partially supported by JSPS Grant-in-Aid for Scientific Research (C) 21K03211. }

\begin{document}

\maketitle

\setlength{\baselineskip} {15.3pt}

\begin{abstract}
Ulrich ideals in numerical semigroup rings of small multiplicity are studied. If the semigroups are three-generated but not symmetric, the semigroup rings are Golod, since the Betti numbers of the residue class fields of the semigroup rings form an arithmetic progression; therefore, these semigroup rings are G-regular (\cite{greg}), possessing no Ulrich ideals. Nevertheless, even in the three-generated case, the situation is different, when the semigroups are symmetric. We shall explore this phenomenon, describing an explicit system of generators, that is the normal form of generators,  for the Ulrich ideals in  the numerical semigroup rings of multiplicity at most 3. As the multiplicity is greater than $3$, in general the task of determining all the Ulrich ideals seems formidable, which we shall experience, analyzing one of the  simplest examples of semigroup rings of multiplicity $4$.
\end{abstract}


{\footnotesize \tableofcontents}

\section{Introduction}\label{intro}

This paper aims at the  study of Ulrich ideals in numerical semigroup rings of small multiplicity.

\vspace{0.3em}

Let $H$ be a numerical semigroup and let $k[[H]]$ be the semigroup ring over a field $k$. Let $\calX_{k[[H]]}$ denote the set of Ulrich ideals in $k[[H]]$. We then naturally attain the following. 
\begin{prob}\label{p}
Determine the set $\calX_{k[[H]]}$.
\end{prob}
\noindent
The present purpose is to report a few partial solutions for Problem \ref{p}, especially in the case where $H$ has small multiplicity $\rme(H)= \min~[H \setminus \{0\}]$.

The notion of Ulrich ideal on which we focus throughout this paper is one of the modifications of that of {\it stable} maximal ideal introduced  in 1971 by his famous paper \cite{L} of J. Lipman. The present modification was formulated by S. Goto, K. Ozeki, R. Takahashi, K.-i. Watanabe, and K.-i. Yoshida \cite{GOTWY} in 2014, where the authors developed the basic theory, revealing that the Ulrich ideals of Cohen-Macaulay local rings enjoy a beautiful structure theorem  for minimal free resolutions.

In order to explain our aim as well as our motivation more precisely, let $(A, \m) $ be a Cohen-Macaulay local ring with $\dim A=d \ge 0$ and $I$ an $\m$-primary ideal of $A$. We throughout assume that $I$ contains a parameter ideal $Q$ of $A$ as a reduction; hence $I^{n+1} = QI^n$ for some $n \ge 0$. This assumption is  naturally satisfied, if the residue class field $A/\m$ is infinite, or if $A$ is analytically irreducible and of dimension one; for example, $A = k[[H]]$,  the semigroup ring of a numerical semigroup $H$ over a field $k$.

\begin{defn} (\cite[Definition 1.1]{GOTWY})\label{2.1}
We say that $I$ is  an {\it Ulrich} ideal of $A$, if the following conditions are satisfied.
\begin{enumerate}[$(1)$]
\item $I \ne Q$,  $I^2=QI$ and
\item $I/I^2$ is a free $A/I$-module.
\end{enumerate}
\end{defn}

\noindent
We notice that Condition $(1)$ of the definition is satisfied if and only if the associated graded ring $\gr_I(A) = \bigoplus_{n\ge 0} I^n/I^{n+1}$ is a Cohen-Macaulay ring with $\rma(\gr_I(A))=1-d$, where $\rma(\gr_I(A))$ denotes the $\rma$-invariant of $\gr_I(A)$ (\cite[Definition 3.1.4]{GW}). Thus, Condition $(1)$ is independent of the choice of reductions $Q$ of $I$. When $I=\m$, Condition $(2)$ is automatically satisfied, while Condition $(1)$ is equivalent to saying that $A$ is not regular but of minimal multiplicity. One finds the general basic results on Ulrich ideals in the fundamental paper \cite{GOTWY}. For example, provided $I, J$ are Ulrich ideals of $A$, $I=J$ if and only if for some $i \ge 0$ $$\operatorname{Syz}_A^i(I) \cong \operatorname{Syz}_A^i(J)$$
as an $A$-module, where $\operatorname{Syz}_A^i(I)$ (resp. $\operatorname{Syz}_A^i(J))$ stands for the $i$-th syzygy module of $I$ (resp. $J$) in a minimal free resolution of the $A$-module $I$ (resp. $J$).

Let $I$ be an $\m$-primary ideal of $A$ and assume that $I^2=QI$. Then, since $Q/QI$ is a free $A/I$-module of rank $d$, the exact sequence
$$
0 \to Q/QI \to I/I^2 \to I/Q \to 0
$$
of $A/I$-modules readily shows that $I/I^2$ is a free $A/I$-module if and only if so is $I/Q$. Therefore, provided that $I$ is minimally generated by $d+1$ elements, the latter condition is equivalent to saying that $I/Q \cong A/I$ as an $A/I$-module, or equivalently $Q:_AI =I$. On the other hand, if $I$ is an Ulrich ideal, by \cite{GOTWY, GTT2} we get the equality 
$$
(\mu_A(I)-d)\cdot\rmr(A/I)= \rmr(A),
$$
where $\mu_A(I)$ (resp. $\rmr(*)$) denotes the number of generators of $I$ (resp. the Cohen-Macaulay type). Therefore
$$d+1 \le \mu_A(I) \le d + \rmr(A),$$
so that when $A$ is a Gorenstein ring, that is the case where $\rmr(A)=1$, every Ulrich ideal $I$ is generated by $d+1$ elements (if it exists), whence $I$ is a {\it good} ideal of $A$ in the sense of \cite{GIW}. As is shown in \cite{GOTWY, GTT2}, all the Ulrich ideals with the extreme number $\mu_A(I)=d+1$ of generators possess finite G-dimension, and their minimal free resolutions have a very restricted form, so that they are eventually periodic of period one.

Let us explain this phenomenon in the case of dimension one. We now assume that $(A, \m)$ is a Cohen-Macaulay local ring with $\dim A=1$ and that $I$ is an Ulrich ideal of $A$. Therefore, $I$ is an $\m$-primary ideal of $A$, and $I^2 = aI$ for some $a \in I$, such that $I \ne (a)$ but $I/(a)$ is a free $A/I$-module.  We assume that $I$ is minimally generated by {\it two} elements, say $I=(a, b)$ with $b \in I$, and write $b^2 = ac$ for some $c \in I$. We then have, since $I/(a) \cong A/I$, that $(a):_Ab = I$, and the minimal free resolution of $I$ has the following form
$$
\ \ \cdots \longrightarrow A^{\oplus 2} \overset{
\begin{pmatrix}
-b & -c\\
a & b
\end{pmatrix}}{\longrightarrow}
A^{\oplus 2} \overset{
\begin{pmatrix}
-b & -c\\
a & b
\end{pmatrix}}{\longrightarrow}A^{\oplus 2} \overset{\begin{pmatrix}
a & b
\end{pmatrix}}{\longrightarrow} I \longrightarrow 0
$$
(\cite[Example 7.3]{GOTWY}). In particular,  $I$ is a {\it totally reflexive} $A$-module, that is $I$ is a  reflexive $A$-module, $\Ext_A^p(I, A) =(0)$, and $\Ext_A^p(\Hom_A(I, A), A) = (0)$ for all $p >0$ (\cite[Proposition 4.6]{GIT}). We then clearly have that $I =J$, once $$\operatorname{Syz}_A^i(I) \cong \operatorname{Syz}_A^i(J)$$ for some $i \ge 0$, provided $I,J$ are Ulrich ideals of $A$.

It seems reasonable to expect that behind the behavior of Ulrich ideals and their existence also, there is hidden some ample information about the structure of the base rings. For example, if $A$ has finite Cohen-Macaulay representation type, then $A$ contains only finitely many Ulrich ideals  (\cite{GOTWY}). In a one-dimensional non-Gorenstein almost Gorenstein local ring, the only possible Ulrich ideal is the maximal ideal (\cite[Theorem 2.14]{GTT2}). In \cite{GIT}  the authors explored the ubiquity of Ulrich ideals in a {\it $2$-AGL} rings (one of the generalizations of Gorenstein local rings of dimension one), and showed that the existence of two-generated Ulrich ideals provides a rather strong restriction on the structure of the base local rings (\cite[Theorem 4.7]{GIT}). Nevertheless, even for the one-dimensional Cohen-Macaulay local rings, in general we lack an explicit and physical list of Ulrich ideals contained inside those rings, which possibly prevents further developments of the study of Ulrich ideals. In order to supply the lack, continuing the work \cite{GIT}, the present research particularly focus on and investigate the question of how many and how ample two-generated Ulrich ideals are contained in a given {\it numerical semigroup ring}, which is  a prototype of Cohen-Macaulay local rings of dimension one.  As we shall  show in the following, although the task is rather tough and the statements of the results are seemingly complicated, we are able to describe {\it all} the Ulrich ideals in certain specific numerical semigroup rings. The list which we will give could enrich the known class of Ulrich ideals, providing numerous concrete examples of totally reflexive modules, as well.

In order to explain how this paper is organized, we turn our attention to the following specific setting. Let $a_1, a_2, \ldots, a_\ell \in \Bbb Z$ be positive integers such that $\mathrm{GCD}~(a_1, a_2, \ldots, a_\ell)=1$. We set
$$
H = \left<a_1, a_2, \ldots, a_\ell\right>=\left\{\sum_{i=1}^\ell c_ia_i ~\middle|~ 0 \le c_i \in \Bbb Z~\text{for~all}~1 \le i \le \ell \right\}
$$
and call it the {\it numerical semigroup} generated by $\{a_i\}_{1 \le i \le \ell}$. The reader may consult the book \cite{RG} for the fundamental results on numerical semigroups. Let $V = k[[t]]$ denote the formal power series ring over a field $k$, and set 
$$
k[[H]] = k[[t^{a_1}, t^{a_2}, \ldots, t^{a_\ell}]] \subseteq V,
$$
which we call the {\it semigroup ring} of $H$ over $k$. The ring  $A=k[[H]]$ is a Noetherian  integral local domain and  $V$ is a birational module-finite extension of $A$, so that  $\overline{A} = V$ (here $\overline{A}$ denotes the integral closure of $A$ in its quotient field), $\dim A = 1$, and the maximal ideal of $A$ is given by $\m = (t^{a_1},t^{a_2}, \ldots, t^{a_\ell} )$. Let 
$$
\rmc(H) = \min \{n \in \Bbb Z \mid m \in H~\text{for~all}~m \in \Bbb Z~\text{such~that~}m \ge n\}
$$
and set $\rmf(H) = \rmc(H) -1$. We then have $A:V = t^{\rmc(H)}V$, and 
$$\rmf(H) = \max ~(\Bbb Z \setminus H)$$ which is called the {\it Frobenius number} of $H$. Let $$\operatorname{e}(H)=\min~[H \setminus \{0\}].$$ Notice that $\operatorname{e}(H)$ coincides with the multiplicity of $A$ with respect to $\m$. Let $\calX_A$ stand for the set of Ulrich ideals in $A$. The ring $A=k[[H]]$ contains only finitely many Ulrich ideals generated by monomials in $t$ (\cite{GOTWY}), and naturally, the present research is more interested in Ulrich ideals which are not generated by monomials in $t$.

With this notation, in Section $2$ we summarize some basic properties of Ulrich ideals in $A$. It is rather difficult to pinpoint the members of $\calX_A$, and to overcome the difficulty, we need a new method to make the list of Ulrich ideals, which we will discuss in Section 2 (Theorem \ref{2.8}). Section 3 is devoted to make a complete list of $\calX_A$ in the case where $\operatorname{e}(H)=3$ (Theorem \ref{e=3}). Our proof is elementary, but rather long, so that it  will be divided into several steps. In Section $4$, we explore Ulrich ideals in the numerical semigroup ring $A =k[[t^4,t^{13}]]$ (Theorem \ref{e=4}). As is well-known, three-generated non-symmetric numerical semigroups are of a special kind (\cite{Herzog}). We will show in Section 5 that for every three-generated non-symmetric numerical semigroup $H$, the Betti numbers of the residue class field of the ring $A=k[[H]]$ form an arithmetic progression, whence $A$ is a Golod ring (Corollary \ref{4.2}),  so that $\calX_A= \emptyset$ (\cite{greg}). On the other hand, every three-generated symmetric numerical semigroup is obtained by {\it gluing}.  Thanks to this fact, we shall partially answer in Section 5, for three-generated symmetric numerical semigroups $H$, the question of whether $\calX_{k[[H]]}$ is empty or not (Proposition \ref{5.6}).

\section{Two-generated Ulrich ideals in core subrings of $V=k[[t]]$}

Let $k$ be a field and let $V = k[[t]]$ denote the formal power series ring over $k$. Let $A$ be a $k$-subalgebra of $V$. Then, following \cite{EGMY}, we say that $A$ is a {\it core}  of $V$, if $t^c V \subseteq A$ for some $c \gg 0$. The semigroup rings $k[[H]]$ of numerical semigroups $H$ are typical examples of cores of $V$. Nevertheless, cores of $V$ do not necessarily arise as semigroup rings. Let us note one of the simplest examples.

\begin{ex}[cf. {\cite[Example 2.1]{EGMY}}]\label{Example 2.1} 
Let $A = k[t^2 + t^3] + t^4V$. Then $A \ne k[[H]]$ for any numerical semigroup $H$.
\end{ex}
\noindent
If $I$ is an Ulrich ideal in the semigroup ring $A$ of a numerical semigroup, the blowing-up ring $A^I = \bigcup_{n \ge 0}[I^n:I^n]$ of $A$ with respect to $I$ is again a core of $V$, which is, however, not necessarily a semigroup ring. In despite of the disadvantage, inside the core $A^I$  there is contained ample information about the characteristic of $I$, which might enable us, for example, to describe a precise  system of generators of $I$. Keeping this anticipation, we shall summarize below some preliminary results about cores and their Ulrich ideals.

Let $A$ be a core of $V$ and suppose that $t^{c_0} V \subseteq A$ with an integer $c_0 \gg 0$. We then have  
$$
k[[t^{c_0}, t^{c_0+1}, \ldots, t^{2c_0-1}]] \subseteq A \subseteq V, 
$$
so that $V$ is a birational module-finite extension of $A$. Hence, $V = \overline{A}$, and $A$ is a one-dimensional Cohen-Macaulay integral complete local domain. We have $V/\n \cong A/\m$, where $\m$ (resp. $\n=tV$) stands for the maximal  ideal of $A$ (resp. $V$). Let $\rmo(*)$ denote the $\n$-adic valuation (or the order function) of $V$ and set
$$
v(A)= \{\rmo(f) \mid 0 \ne f \in A\}.
$$
Then, $H=v(A)$ is called the {\it value semigroup} of $A$, which is indeed a numerical semigroup, because $c_0, c_0+1 \in H$. Let $\fkc = A:V$ denote the conductor of $A$. Then, $\fkc = t^{\rmc(H)}V$, since $t^{c_0}V \subseteq \fkc$.  We have $\rme(H) = \rme(A)$, where $\rme(A)$ denotes the multiplicity of $A$ with respect to $\m$.


\begin{setting}
Let $I$ be a fixed two-generated  Ulrich ideal of $A$. Let $f, g \in I$ such that $I = (f, g)$ and $I^2=fI$. We consider the $A$-subalgebra $$A^I=\bigcup_{n \ge 0}[I^n : I^n]$$
of $V$, where the colon $$I^n:I^n =\{x \in \operatorname{Q}(A) \mid xI^n \subseteq I^n \}$$ is considered inside the quotient field of $A$. We then have $A^I = I:I$ since $I^{n+1}=f^{n}I$ for all $n \ge 0$, so that $$A^I = f^{-1}I =A + A{\cdot}\frac{g}{f}.$$  We set $a = \rmo(f)$, $b = \rmo(g)$, and $c = \rmc(H)$. Notice that $a$ is an invariant of $I$, since $IV=fV=t^aV$ (see Lemma \ref{2.4} also).
\end{setting}

\begin{lem}\label{2.3}
$t^{c - (b-a)} V \cap A \subseteq I$ and $\fkc \subseteq I$. 
\end{lem}

\begin{proof}
We have 
$
A^I = A + A\frac{g}{f}
$ in $V$. Therefore, because $\frac{g}{f}V =t^{b-a}V$ and $\fkc = t^cV$, we get
$$
\left(t^{c-(b-a)} V \cap A\right)\frac{g}{f} \subseteq \left(t^{c-(b-a)} V\right)\frac{g}{f} = t^cV =\fkc \subseteq A.
$$
Hence,
$$t^{c-(b-a)} V \cap A \subseteq A:_A\frac{g}{f} = (f):_Ag = I,$$ where the last equality follows from the fact that $I/(f) \cong A/I$ as an $A/I$-module. Since $\fkc = t^cV \subseteq t^{c - (b-a)} V \cap A$, the second assertion follows.
\end{proof}

\begin{lem}\label{2.4}
$
a = 2{\cdot}\ell_A(A/I).
$
\end{lem}

\begin{proof}
Let $\rme^0_{(f)}(M)$ denote, for each finitely generated $A$-module $M$,  the multiplicity  of $M$  with respect to the parameter ideal $(f)$ of $A$. Then, $\rme^0_{(f)}(I)=\ell_A(I/fI)$, since $I$ is a maximal Cohen-Macaulay $A$-module (here $\ell_A(*)$ denotes the length), while $$\rme^0_{(f)}(I)=\rme^0_{(f)}(A)=\rme^0_{(f)}(V) = \rmo(f).$$ Therefore, $a =2{\cdot}\ell_A(A/I)$, because $I/fI \cong (A/I)^{\oplus 2}$ as an $A$-module.
\end{proof}

We furthermore have the following.

\begin{lem}\label{lemma0}
The following assertions hold true.
\begin{enumerate}[$(1)$]
\item $0 < a \le b < a + c$.
\item $2b-a \in H$.
\item If $a \ge c$, then $\rme(A) = 2$ and $I=\fkc$.
\end{enumerate}
\end{lem}

\begin{proof}
Because $fV = IV$, we have $\frac{g}{f} \in V$, whence $b-a \ge 0$. Since $I$ is minimally generated by $f$ and $g$, we get $g \not\in fA$, whence $\frac{g}{f} \not\in \fkc$. Thus $b-a = \rmo(\frac{g}{f}) < c$. Therefore
$$
0 < a \le b < a + c. 
$$
Because $g^2 \in I^2 = fI$, we get $2b-a \in H$. Assume that $a \ge c$. Then $f, g \in \fkc=t^cV$, whence $I \subseteq \fkc$. Consequently, Lemma \ref{2.3} forces $I = \fkc$, so that $I = fV$, because $I$ is an ideal of $V$ and $IV=fV$. Consequently, since $I \cong V$ as an $A$-module, we have $$2=\mu_A(I) = \mu_A(V) = \rme(A)$$
as claimed \end{proof}

Notice that $b-a$ may belong to $H$, since $I = (f, f+g)$. We however have the following.

\begin{prop}\label{2.5}
One can choose the elements $f, g \in I$ so that $b-a \not\in H$.
\end{prop}

\begin{proof}
Suppose that $b-a \in H$ and choose $\xi \in A$ so that $$\rmo(\xi) = b-a, \ \ \rmo(g-f\xi) > b.$$ Set $g_1 = g - f\xi$ and $b_1 = \rmo(g_1)$. We then have $I = (f, g_1)$ and 
\begin{center}
$0 < a \le b < b_1 < a + c$
\end{center}
where the last inequality comes from the facts that $\frac{g_1}{f} \not\in A$ and that $\fkc =t^cV \subseteq A$. If $b_1-a \in H$, let us choose $g_2 \in I$ so that $I = (f, g_2)$ and $b_1 < b_2=\rmo(g_2)$. Since still $b_2 < a + c$, this procedure will terminate after finitely many steps, which shows that we can eventually choose the elements $f,g \in I$ so that $I = (f,g)$, $I^2=fI$, and $b-a \not\in H$.
\end{proof}

We summarize the above arguments in the following.

\begin{thm}\label{2.8}
Let $A$ be a core of $V$ and let $H=v(A)$. Let $I$ be an Ulrich ideal in $A$ with $\mu_A(I) =2$. Then one can choose elements $f, g \in I$ so that the following conditions are satisfied, where $a = \rmo(f)$, $b=\rmo(g)$, and $c = \rmc(H)$.
\begin{enumerate}[$(1)$]
\item $I=(f, g)$ and $I^2= fI$. 
\item $a, b \in H$ and $0 < a < b < a + c$.
\item $b-a \not\in H$, $2b-a \in H$, and $a = 2\cdot \ell_A(A/I)$.
\item If $a \ge c$, then $\rme(A) = 2$ and $I=\fkc$.
\end{enumerate}
\end{thm}

Let us give the following, showing how Theorem \ref{2.8} works to pinpoint the elements of the set $\calX_A$ of Ulrich ideals in $A$.

\begin{ex}[{cf. \cite[Theorem 1.7]{GOTWY}, \cite[Theorem 6.1]{GIT}}]\label{ex0}
$\calX_{k[[t^3,t^4]]}= \{(t^4,t^6)\}$ and $\calX_{k[[t^3,t^5]]}= \emptyset$.
\end{ex}

\begin{proof}
Let $A =k[[t^3, t^4]]$ and $I= (t^4,t^6)$. Then, $I^2=t^4I$ and $\fkc = (t^6,t^7, t^8) \subseteq I$. We have $\ell_A(A/I)\le 2$, whence $\ell_A(I/(t^4))\ge 2$ because $\ell_A(A/(t^4))=4$, so that the epimorphism $$\varphi : A/I \to I/(t^4), \ \ \varphi(1~\mod~I)=t^6~\mod~(t^4)$$ of $A$-modules must be an isomorphism. Hence, $I=(t^4,t^6)$ is an Ulrich ideal in $A$. Conversely, let $I \in \calX_A$ and choose $f,g \in I$ so that all the conditions stated in Theorem \ref{2.8} are satisfied. Then, $b-a = 1,2$, or $5$, while $a$ is even. If $a \ge c=6$, then $I=\fkc =(t^6, t^7, t^8)$ by Lemma \ref{lemma0}, so that $\mu_A(I) = 3$, since $I \cong V$ as an $A$-module. This is impossible, because $I$ is two-generated. Therefore, $a= 4$, and $b-a= 1,2$, or $5$. We consider the following table.

\vspace{-0.3em}
$$
\begin{tabular}{|c|c|c|c|}
\hline
$a$ & $4$  & $4$ & $4$ \\ \hline 
$b-a$ & $1$  & $2$ & $5$  \\ \hline 
$b$ & $5$ &  $6$ & $9$ \\ \hline
$2b-a$ & $6$ &  $8$ & $14$ \\ \hline
$6-(b-a)$ & $5$ &  $4$ & $1$ \\ \hline
\end{tabular} \vspace{1em}
$$
Here, the values of the second (resp. the third and the fourth) column indicate the possible values of $b, 2b-a$, and $6 -(b-a)$, when $a =4$ and $b-a = 1, 2, 5$, respectively.  We then have $b-a \ne 1$, since $5 \not\in H$. Suppose that $b-a= 5$. Since $6 -(b-a) =1$, we get $t^3 \in I=(f,g)$ by Lemma \ref{lemma0}, which is impossible, because $\rmo(f)=4$ and $\rmo(g)=9$. Therefore, $a=4$ and $b=6$. Consequently, 
$f= t^4 + \rho$ with $\rho \in \fkc$ and $g \in \fkc$, so that $$I=(f,g)+(t^6, t^7, t^8)=(t^4, t^6,t^7,t^8),$$ because $I \supseteq \fkc$ by Lemma \ref{2.3}. Thus, $I = (t^4, t^6)$, whence $\calX_A=\{(t^4,t^6)\}$.

We similarly conclude that $\calX_{k[[t^3,t^5]]}= \emptyset$, whose proof we would like to leave to the reader.
\end{proof}

\begin{ex}\label{2.9}
Let $A = k[[t^2, t^{2\ell + 1}]]~(\ell \ge 1)$ . Then 
$$
\calX_A = \{(t^{2q}, t^{2\ell + 1}) \mid 1 \le q \le \ell\}. 
$$ 
\end{ex}

\begin{proof}
We have $A = k[[H]]$ for the semigroup $H = \left<2, 2\ell + 1\right>$. Hence, $c = 2\ell$ and  $\fkc = t^{2 \ell}V=(t^{2\ell}, t^{2\ell + 1})$. Let  $I = (t^{2q}, t^{2\ell + 1})$ ($1 \le q  \le \ell$). We set $x=t^2$ and $y = t^{2\ell + 1}$. Then, since $I=(x^q,y)$ and since $$y^2 = x^{2 \ell + 1} = x^q \cdot x^{2 \ell + 1 -q} \in x^q I,$$ we get $I^2 = x^q I$, while $A/I \cong I/(x^q)$ as an $A$-module, because $I/(x^q)$ is a homomorphic image of $A/I$ and 
$$\ell_A(A/I)= \ell_A(I/(x^q))= q. 
$$  Therefore, $I$ is an Ulrich ideal of $A$.

Conversely, let $I$ be an Ulrich ideal of $A$, and choose $f, g \in I$ so that the conditions stated in Theorem \ref{2.8} are satisfied. Then, $0 <b-a \notin H$ and $a \ge 2$ is even. Hence, $b-a  \in \{1,3,5, \cdots, 2\ell-1\}$, so that $b \in H$ and $b$ is odd. Therefore, $b \ge 2\ell + 1$ whence $g \in t^{2\ell} V = \fkc$, so that $$I= (f,g)=(f)+\fkc=(f, t^{2 \ell}, t^{2\ell + 1})$$ where the second equality follows from Lemma \ref{2.3}. Let us write $a = 2q$ ($q \ge 1$). If $q \ge \ell$, then  $f \in \fkc$ also, so that $I = \fkc = (t^{2\ell}, t^{2\ell + 1})$. Assume $q < \ell$, and write $$f = t^{2q} + c_{q+1}t^{2(q + 1)} + c_{q+2}t^{2(q + 2)} + \cdots + c_{\ell-1}t^{2(\ell-1)} + \rho$$ where $c_i \in k$ ($q+1 \le i \le \ell -1$) and $\rho \in \fkc = t^{2\ell}V$. Then, since $1-c_{q+1}t^2$ is invertible in $A$, by replacing $f$ with $f-c_{q+1}t^2f=(1-c_{q+1}t^2)f$ if necessary, we may assume that $c_{q+1} = 0$. By repeating the same procedure for the remaining coefficients $c_i$'s in $f$, we finally obtain
$$
I = (t^{2q} + \rho, t^{2\ell}, t^{2\ell + 1}) =  (t^{2q}, t^{2 \ell}, t^{2\ell + 1}) =  (t^{2q}, t^{2\ell + 1})
$$
as claimed. 
\end{proof}

\begin{remark}
In Example \ref{2.9}, let $1 \le q_1, q_2 \le \ell$. Then $(t^{2q_1}, t^{2\ell + 1})=(t^{2q_2}, t^{2\ell + 1})$ if and only if $q_1=q_2$. 
\end{remark}


\section{Numerical semigroup rings of multiplicity $3$}
Let $H$ be a numerical semigroup with $\rme(H)=3$ and let $A=k[[H]]$ be the semigroup ring of $H$ over a field $k$.

The purpose of this section is to determine all the {\it two-generated} Ulrich ideals in $A$. Since $\rme(H)=3$, $H$ is at most three-generated, and if $H$ is minimally three-generated, the Cohen-Macaulay local ring $A$ has minimal multiplicity, so that $A$ is G-regular (\cite{greg}), containing no two-generated Ulrich ideals. This observation allows us to assume that $H= \left<3, \ell \right>$, where $\ell \ge 4$ is an integer such that $\operatorname{GCD}(3,\ell)=1$. Thanks to Example \ref{ex0}, we may assume that $\ell \ge 7$.

The goal of this section is Theorem \ref{e=3} below. With the following setting we divide the proof into several steps.

\begin{setting}
Let $I$ be an Ulrich ideal of $A=k[[t^3, t^{\ell}]]$ with $\ell \ge 7$ and $\operatorname{GCD}(3,\ell)=1$. Hence, $\mu_A(I) = 2$, because $A$ is a Gorenstein ring. We choose elements $f, g \in I$, so that all the conditions stated in Theorem \ref{2.8} are satisfied. We set $a = \rmo(f)$, $b = \rmo(g)$, and $c =\rmc(H)$. Let $$B= \bigcup_{n\geq 0}\left[I^n: I^n\right], \ \ \xi = \frac{g}{f},\ \ \text{and} \ \ H_1=v(B).$$ Then $B = f^{-1}I=A + A\xi$ with $\mu_A(B)=2$ and $B=k[[t^3,t^{\ell}, \xi]]$ is a Gorenstein local ring (\cite[Corollary 2.6 (b)]{GOTWY}). We have $\rmo(\xi)=b-a$, whence $b-a \in H_1 \setminus H$.  
\end{setting}

Let us note the following.

\begin{lem}\label{2.1}
$\xi \not\in \m B$ but $\xi^2 \in \m B$.
\end{lem}

\begin{proof}
Let $\m_B$ be the maximal ideal of $B$. Then, since $\ell_A(B/\m B)=2$ and $B = A + A\xi$, we have $\xi \not\in \m B$ and $(\m_B/\m B)^2=(0)$. Since the maximal ideal $\m_B/\m B$ of $B/\m B$ is principal and generated by the image $\overline{\xi}$ of $\xi$,  we get  ${\overline{\xi}}^2 = 0$ in $B/\m B$.
\end{proof}

\begin{lem}\label{2.2}
Suppose that $H_1 = \left<3, \alpha \right>$ for some $\alpha>0$. Then, $\frac{\ell}{2} \le \alpha < \ell$. 
\end{lem}

\begin{proof}
As $H_1 \supsetneq H$, we have $\alpha \not\in H$. Hence, $\alpha < \ell$, because $\ell \in H_1=\left<3, \alpha \right>$ and $\ell \not\in 3\bbZ$. Choose $\eta \in B$ so that $\rmo(\eta)=\alpha$. We then have $B = k[[t^3, \eta]]$. In fact, let $C = k[[t^3, \eta]]$. Then, $t^nV \subseteq C$ for some $n \gg 0$, and therefore, because $$C \subseteq B   \ \  \text{and} \ \ v(B) = \left<3,\alpha \right> \subseteq v(C),$$ we naturally get $B \subseteq C$, whence $B = k[[t^3,\eta]]$ and $\eta^2 \in \m B=(t^3, t^{\ell})B$ (see the proof of Lemma \ref{2.1}). Consequently, if $2 \alpha < \ell$, then passing to the expression
$$ \eta^2 = t^3 \varphi + t^{\ell}\psi$$
of $\eta^2$ with $\varphi, \psi \in B$, we get $2\alpha -3 = \rmo(\varphi) \in H_1$, which is impossible, because $\rmc(H_1) 
= (3-1)(\alpha-1) = 2\alpha -2$. Thus, $2\alpha \ge \ell$.
\end{proof}

Since $\mu_A(V) = 3$, we have $B \ne V$, whence $1 \not\in H_1$. Therefore, if $2 \in H_1$, then $H_1 = \left<2, 3\right>$, so that Lemma \ref{2.2} forces $\ell \le 6$, which violates the assumption that $\ell \ge 7$. Thus, $2 \not\in H_1$. Consequently, $\rme(H_1)=\min [H_1 \setminus\{0\}] = 3$, and $H_1$ is symmetric by \cite{Kunz}, because $B$ is a Gorenstein ring. Hence $$H_1 = \left<3, \alpha\right> \  \text{for~some}~\alpha \ge 4~\text{such~that}~\alpha \not\equiv 0~\mod~3.$$ We furthermore have the following.

\begin{prop}\label{3.3}
$\alpha = b-a$. Hence $H_1=\left<3,b-a\right>$ and $B = k[[t^3, \xi]]$.
\end{prop}

\begin{proof}
Since $b-a \in H_1 =\left<3,\alpha \right>\setminus H=\left<3, \ell\right>$, we have $b-a \ge \alpha$. We choose $\eta \in B$, so that $\rmo(\eta) = \alpha$. Then, $\eta \in \m_B= (t^3, t^{\ell}, \xi)B$. Therefore, if $b-a > \alpha$, passing to the expression 
$$\eta  = t^3\varphi + t^{\ell}\psi + \xi\delta $$
with $\varphi, \psi, \delta \in B$, we have $\rmo(\varphi) = \alpha - 3 \in H_1=\left<3,\alpha \right>$ since $\alpha < \ell$ by Lemma \ref{2.2}). Therefore, $\alpha \equiv 0 ~\mod~3$, which is absurd. Thus $\alpha = b-a$. See the proof of Lemma \ref{2.2} for the equality $B =k[[t^3,\xi]]$. 
\end{proof}

Here let us draw the shape of the semigroup $H=\left<3,\ell\right>$. The figures might be helpful for the reader to grasp the arguments making progress below. In the following figure, the numbers located in the gray part describe the elements of $H = \left<3, \ell \right>$, according to the four cases: $\ell = 3n+ 1$ where $n$ is odd, $\ell = 3n+ 1$ where $n$ is even, $\ell = 3n+ 2$ where $n$ is odd, and $\ell = 3n+ 2$ where $n$ is even, respectively.

$$
 \ytableausetup{mathmode, boxsize=2.5em}
 \begin{ytableau}
 *(lightgray) 0 & 1 &  2 \\
 *(lightgray) 3 & 4 & 5  \\
  \none[\vdots] & \none[\vdots] & \none[\vdots] \\
 *(lightgray) \scriptstyle 6q   &  \scriptstyle 6q+1 &  \scriptstyle 6q+2\\ 
 *(lightgray) \scriptstyle 6q+3  &   *(lightgray)  \scriptstyle \ell=6q+4 & \scriptstyle 6q+5\\
 *(lightgray)  \scriptstyle 6q+6    &   *(lightgray) \scriptstyle 6q + 7 & \scriptstyle 6q+8\\
     \none[\vdots] & \none[\vdots] & \none[\vdots] \\
      *(lightgray)  \scriptstyle 12q+3  &  *(lightgray)  \scriptstyle 12q+4 &  \scriptstyle 12q+5\\
 *(lightgray)  \scriptstyle 12q+6  &   *(lightgray) \scriptstyle 12q+7 &  *(lightgray) \scriptstyle 12q+8\\
 \end{ytableau} 
  \ \ \ \ \ \ \   
 \ytableausetup{mathmode, boxsize=2.5em}
 \begin{ytableau}
 *(lightgray) 0 & 1 &  2 \\
 *(lightgray) 3 & 4 & 5  \\
  \none[\vdots] & \none[\vdots] & \none[\vdots] \\
 *(lightgray) \scriptstyle 6q-3   &  \scriptstyle 6q-2 &  \scriptstyle 6q-1\\ 
 *(lightgray) \scriptstyle 6q  &   *(lightgray)  \scriptstyle \ell=6q+1 & \scriptstyle 6q+2\\
 *(lightgray)  \scriptstyle 6q+3    &   *(lightgray) \scriptstyle 6q + 4 & \scriptstyle 6q+5\\
     \none[\vdots] & \none[\vdots] & \none[\vdots] \\
      *(lightgray)  \scriptstyle 12q-3  &  *(lightgray)  \scriptstyle 12q-2 &  \scriptstyle 12q-1\\
 *(lightgray)  \scriptstyle 12q  &   *(lightgray) \scriptstyle 12q+1 &  *(lightgray) \scriptstyle 12q+2\\
 \end{ytableau} 
  \ \ \ \ \ \ \ 
 \ytableausetup{mathmode, boxsize=2.5em}
 \begin{ytableau}
 *(lightgray) 0 & 1 &  2 \\
 *(lightgray) 3 & 4 & 5  \\
  \none[\vdots] & \none[\vdots] & \none[\vdots] \\
 *(lightgray) \scriptstyle 6q   &  \scriptstyle 6q+1 &  \scriptstyle 6q+2\\ 
 *(lightgray) \scriptstyle 6q+3  &     \scriptstyle 6q+4 & *(lightgray)\scriptstyle \ell=6q+5\\
 *(lightgray)  \scriptstyle 6q+6    &    \scriptstyle 6q + 7 & *(lightgray) \scriptstyle 6q+8\\
     \none[\vdots] & \none[\vdots] & \none[\vdots] \\
      *(lightgray)  \scriptstyle 12q+6  &    \scriptstyle 12q+7 & *(lightgray)  \scriptstyle 12q+8\\
 *(lightgray)  \scriptstyle 12q+9  &   *(lightgray) \scriptstyle 12q+10 &  *(lightgray) \scriptstyle 12q+11\\
 \end{ytableau}  
  \ \ \ \ \ \ \ 
 \ytableausetup{mathmode, boxsize=2.5em}
 \begin{ytableau}
 *(lightgray) 0 & 1 &  2 \\
 *(lightgray) 3 & 4 & 5  \\
  \none[\vdots] & \none[\vdots] & \none[\vdots] \\
 *(lightgray) \scriptstyle 6q-3   &  \scriptstyle 6q-2 &  \scriptstyle 6q-1\\ 
 *(lightgray) \scriptstyle 6q  &     \scriptstyle 6q+1 & *(lightgray) \scriptstyle \ell=6q+2\\
 *(lightgray)  \scriptstyle 6q+3    &   \scriptstyle 6q + 4 & *(lightgray)\scriptstyle 6q+5\\
     \none[\vdots] & \none[\vdots] & \none[\vdots] \\
      *(lightgray)  \scriptstyle 12q  &  \scriptstyle 12q+1 &  *(lightgray)   \scriptstyle 12q+2\\
 *(lightgray)  \scriptstyle 12q+3  &   *(lightgray) \scriptstyle 12q+4 &  *(lightgray) \scriptstyle 12q+5\\
 \end{ytableau} 
 $$
\vspace{0.1em}
\begin{center}
$\ell = 3n+ 1$, $n$ is odd, \ \   $\ell = 3n+ 1$, $n$ is even, \ \  $\ell = 3n+ 2$, $n$ is odd, \ \  $\ell = 3n+ 2$, $n$ is even,
\end{center}
 \ \ \ and \  $q = \frac{n-1}{2}$ \  \ \  \  \ \ \ \ \  \  \  \ \ and  \  $q = \frac{n}{2}$ \ \ \ \ \ \ \ \ \ \ \ \ \ \ \    and \  $q = \frac{n-1}{2}$ \ \ \ \ \ \ \ \ \ \ \  \ \ and \  $q = \frac{n}{2}$

\vspace{1em}

The proof of Assertion (3) (resp. Assertion (4)) in the following lemma is similar to that of Assertion (2) (resp. Assertion (1)). Let us include brief proofs.

\begin{lem}\label{3.4}
The following assertions hold true. 
\begin{enumerate}[$(1)$]
\item Suppose that $\ell = 3n + 1$ where $n \ge 3$ is odd. Let $q = \frac{n-1}{2}$.
\begin{enumerate}
\item[$(\rm i)$] If $\alpha \equiv 1$ $\mod$ $3$, then $\alpha = 3q+1 + 3j$ for some $1 \le j \le q$. 
\item[$(\rm ii)$] If $\alpha \equiv 2$ $\mod$ $3$, then $\alpha = 3q+2$.
\end{enumerate}
\item Suppose that $\ell = 3n + 1$ where $n \ge 2$ is even. Let $q = \frac{n}{2}$. Then $\alpha = 3q+1 + 3j$ for some $0 \le j \le q-1$. 
\item Suppose that $\ell = 3n + 2$ where $n \ge 3$ is odd. Let $q = \frac{n-1}{2}$. Then $\alpha = 3q+2 + 3j$ for some $1 \le j \le q$. 
\item Suppose that $\ell = 3n + 2$ where $n \ge 2$ is even. Let $q = \frac{n}{2}$.
\begin{enumerate}
\item[$(\rm i)$] If $\alpha \equiv 1$ $\mod$ $3$, then $\alpha = 3q+1$.
\item[$(\rm ii)$] If $\alpha \equiv 2$ $\mod$ $3$, then $\alpha = 3q+2 + 3j$ for some $0 \le j \le q-1$. 
\end{enumerate}
\end{enumerate}
\end{lem}

\begin{proof}
$(1)$~$(\rm i)$~This readily follows from the fact that $3q +2 = \frac{\ell}{2} \le \alpha < \ell = 6q+4$, thanks to Lemma \ref{2.2}.  

$(1)$~$(\rm ii)$~We write $\alpha = 3 \beta + 2$. Then $\beta \ge q$, since $3q+2 \le \alpha$. Assume $\beta > q$. Since $\ell = 6q + 4 \in H \subseteq H_1=\left<3, \alpha \right>$, we have $\ell = 3 \varphi + \alpha \psi$ for some $\varphi \ge 0$ and $\psi \ge 1$. If $\psi \ge 2$, then 
$$
\alpha \psi \ge 2 \alpha = 6\beta + 4 > 6q + 4 = \ell \ge \alpha \psi
$$
which is absurd. Hence, $\psi = 1$, and therefore $\ell = 6q+ 4 = 3 \varphi + ( 3 \beta + 2)$, which is impossible. Thus, $\alpha= 3q+2$.

$(2)$~Since $3q + \frac{1}{2} = \frac{\ell}{2} \le \alpha < \ell = 6q + 1$, it suffices to show $\alpha \equiv 1$ $\mod$ $3$. Assume $\alpha \equiv 2$ $\mod$ $3$ and let $\alpha = 3 \beta + 2$. We write $\ell = 3 \varphi + \alpha \psi$ with $\varphi \ge 0$ and $\psi \ge 1$. If $\psi \ge 2$, then 
$$
\alpha \psi \ge 2 \alpha = 6\beta + 4 > 6q + 1 = \ell \ge \alpha \psi
$$
which is absurd. Hence, $\psi = 1$, so that $\ell =6q+1= 3 \varphi + (3 \beta + 2)$,  which is impossible. Thus, $\alpha \equiv 1$ $\mod$ $3$.

$(3)$~We have $3q + \frac{5}{2} = \frac{\ell}{2} \le \alpha < \ell = 6q + 5$ and it suffices to show $\alpha \equiv 2$ $\mod$ $3$. Assume $\alpha \equiv 1$ $\mod$ $3$ and write $\alpha = 3 \beta + 1$. Notice that $\beta > q$, since  $3q + 3 \le \alpha$. Let us write $\ell = 3 \varphi + \alpha \psi$ with $\varphi \ge 0$ and $\psi \ge 1$. If $\psi \ge 2$, then 
$$
\alpha \psi \ge 2 \alpha = 6\beta + 2 \ge 6(q+1) + 2 > \ell \ge \alpha \psi,
$$
which is absurd.  Hence, $\psi = 1$, so that $\ell = 6q+5 =3 \varphi + (3 \beta + 1)$, which is  impossible. Thus, $\alpha \equiv 2$ $\mod$ $3$.

$(4)$~$(\rm i)$~We write $\alpha = 3 \beta + 1$ and assume that $\beta > q$. Let $\ell = 3 \varphi + \alpha \psi$ with $\varphi \ge 0$ and $\psi \ge 1$. If $\psi \ge 2$, then 
$$
\alpha \psi \ge 2 \alpha = 6\beta + 2 > 6q + 2 = \ell \ge \alpha \psi,
$$
which is absurd. Hence, $\psi = 1$, so that $\ell = 6q+ 5=3 \varphi + ( 3 \beta + 1)$, which impossible. Thus, $\alpha= 3q+1$.

$(4)$~$(\rm ii)$~~This readily follows from the fact that $3q + 1 = \frac{\ell}{2} \le \alpha < \ell = 6q+2$. 
\end{proof}

Combining Proposition \ref{3.3} with Lemma \ref{3.4}, we are able to restrict possible semigroups $H_1$.

\begin{prop}\label{3.5}
\begin{enumerate}[$(1)$]
\item Suppose that $\ell = 3n + 1$ where $n \ge 3$ is odd. Let $q = \frac{n-1}{2}$. Then
\begin{center}
\hspace{-1.5em}
$H_1 = \left<3, b-a\right>$ where $b-a = 3q+2$ or $b-a = 3q+1 + 3j$ for some $1 \le j \le q$. 
\end{center}
\item Suppose that $\ell = 3n + 1$ where $n \ge 2$ is even. Let $q = \frac{n}{2}$. Then
\begin{center}
\hspace{-7.8em}
$H_1 = \left<3, b-a\right>$ where $b-a = 3q+1 + 3j$  for some $0 \le j \le q-1$. 
\end{center}
\item Suppose that $\ell = 3n + 2$ where $n \ge 3$ is odd. Let $q = \frac{n-1}{2}$. Then
\begin{center}
\hspace{-9.5em}
$H_1 = \left<3, b-a\right>$ where $b-a = 3q+2 + 3j$  for some   $1 \le j \le q$. 
\end{center}
\item Suppose that $\ell = 3n + 2$ where $n \ge 2$ is even. Let $q = \frac{n}{2}$. Then
\begin{center}
\hspace{-0.2em}
$H_1 = \left<3, b-a\right>$ where $b-a = 3q+1$ or $b-a = 3q+2 + 3j$ for some $0 \le j \le q-1$. 
\end{center}
\end{enumerate}
\end{prop}

The following theorem is the heart of the proof of Theorem \ref{e=3}. The proofs of Assertions (2), (3), and (4) in it  are essentially the same as that of Assertion (1). Nevertheless, because they are subtly different from each other, we would like to note proofs for all of them.

\begin{thm}\label{3.6}
\begin{enumerate}[$(1)$]
\item Suppose that $\ell = 3n + 1$ where $n \ge 3$ is odd. Let $q = \frac{n-1}{2}$. Then
\begin{center}
$(a, b) = (\ell, \ell + 3q + 2)$ or $(a,b) = (6i, \ell + 3i)$ for some $1 \le i \le q$. 
\end{center}
\item Suppose that $\ell = 3n + 1$ where  $n \ge 2$ is even. Let $q = \frac{n}{2}$. Then
\begin{center}
$(a, b) = (6i, \ell + 3i)$ for some $1 \le i \le q$. 
\end{center}
\item Suppose that $\ell = 3n + 2$ where  $n \ge 3$ is odd. Let $q = \frac{n-1}{2}$. Then
\begin{center}
$(a, b) =  (6i, \ell + 3i)$ for some $1 \le i \le q$. 
\end{center}
\item Suppose that $\ell = 3n + 2$ where  $n \ge 2$ is even. Let $q = \frac{n}{2}$. Then
\begin{center}
$(a, b) = (\ell, \ell + 3q + 1)$ or $(a,b=)(6i, \ell + 3i)$ for some $1 \le i \le q$. 
\end{center}
\end{enumerate}
\end{thm}

\begin{proof}
$(1)$~Since $0 < a < c = 2\ell-2 = 12q + 6$ and $a$ is even, we have $a = 6i ~(1 \le i \le 2q)$ or $a = \ell + 6i ~(0 \le i \le q)$. 
We first consider the case where $a = 6i ~(1 \le i \le 2q)$. We look at the following table. 
\vspace{0.3em}
$$
\begin{tabular}{|c|c|c|}
\hline
$a$ & $6i$  & $6i$ \\ \hline 
$b-a$ & $3q+2$  & $3q+1 + 3j~(1 \le j \le q)$  \\ \hline 
$b$ & $3q+2 + 6i$ &  $3q+1 + 6i + 3j~(1 \le j \le q)$ \\ \hline
$2b-a$ & $6q+4 + 6i$ &  $6q+2 + 6i + 6j~(1 \le j \le q)$ \\ \hline
$c-(b-a)$ & $9q + 4$ &  $9q+5 - 3j~(1 \le j \le q)$ \\ \hline
\end{tabular} \vspace{0.3em}
$$
Here, the values of the second (resp. the third) column indicate the possible values of $b, 2b-a$, and $c -(b-a)$, when $a = 6i$ and $b-a = 3q+2$ (resp. $a = 6i$ and $b-a = 3q+1 + 3j$). Our aim is to prove that $(a,b) = (6i, \ell + 3i)$ for some $1 \le i \le q$.

To begin with, we will check that $b-a \ne 3q+2$. Assume the contrary. Then $9q + 6 \in H$ and $b- (9q + 6) = 6(i-q) - 4$. Hence $i \le q$. In fact, if $i > q$, then 
$$
9q + 4-a \in H, 
$$ 
and $9q + 4- a \ge 6q + 4 = \ell$, so that $3q \ge 6i \ge 6(q+1)$. This is absurd. Therefore, $i \le q$. Consequently
$$
b = 3q+2 + 6i \le 9q + 2 < 12 q +5 = c-1,
$$
whence $b \not\in H$, because $b \equiv 2$ $\mod$ $3$. This is absurd and hence $b-a \ne 3q+2$.

Therefore, $b-a =3q+1 + 3j$ for some $1 \le j \le q$, so that
$$
b - (c-(b-a) + 2) = b - (9q+7 - 3j) = 6 (i+j - q - 1).
$$
If $i > q$, then $b > 9q+7 - 3j$. Since $t^{c-(b-a) + 2} \in I$, we have $9q+7 - 3j - 6i \in H$, so that 
$$
9q+7 - 3j - 6i \ge \ell = 6q + 4,
$$
which yields $q \ge j + 2i -1 \ge 2i > 2q$. This is absurd. Thus, $i \le q$.

Because $2b-a \in H$ and $2b-a \equiv 2$ $\mod$ $3$, we have $2b-a \ge 2 \ell = 12 q + 8$. Hence, because $2b-a = 6q+2 + 6i + 6j$, we have $6(i+j - q -1) \ge 0$. In particular, $i+j \ge q+1$. Now assume that $i+j > q+1$. Then, since $b > c-(b-a) + 2$ and $t^{c-(b-a) + 2} \in I$, we get
$$
c-(b-a) + 2 - a \in H \ \ \text{and} \ \ c-(b-a) + 2 - a \equiv 1 \ \mod \ 3
$$
which implies $c-(b-a) + 2 - a \ge \ell = 6q + 4$. Consequently, $q \ge 2i + j-1$, whence $i+j > q + 1 \ge 2i + j$. This is absurd. Thus, $i+j = q + 1$, whence $(a,b)=(6i, \ell + 3i)$ with $1 \le i \le q$.

Secondly,  we consider the case where $a = \ell + 6i ~(0 \le i \le q)$. This time we, have the  following table.
$$
\begin{tabular}{|c|c|c|}
\hline
$a$ & $\ell + 6i$  & $\ell + 6i$ \\ \hline 
$b-a$ & $3q+2$  & $3q+1 + 3j~(1 \le j \le q)$  \\ \hline 
$b$ & $9q+6 + 6i$ &  $9q+5 + 6i + 3j~(1 \le j \le q)$ \\ \hline
$2b-a$ & $12q+8 + 6i$ &  $12q+6 + 6i + 6j~(1 \le j \le q)$ \\ \hline
$\rmc(H)-(b-a)$ & $9q + 4$ &  $9q+5 - 3j~(1 \le j \le q)$ \\ \hline
\end{tabular} \vspace{0.3em}
$$
Suppose $b-a = 3q+2$. We then have 
$$
b- (c-(b-a) + 2) = 6i \in H.
$$ 
If $i > 0$, then $(c-(b-a) + 2) - a \in H$ and $(c-(b-a) + 2) - a < 3q + 2$. This is impossible, because $(c-(b-a) + 2) - a \equiv 2$ $\mod$ $3$. Therefore, $i=0$ and $(a, b)=(\ell, \ell + 3q + 2)$.

Lastly, we assume that $b-a=3q+1 + 3j~(1 \le j \le q)$ and seek a contradiction. Since $i+j > 0$, we have
$$
b- (c-(b-a) + 2) = 6(i+j) -2 >0,
$$
so that $(c-(b-a) + 2)  - a = 3(q-j-2i +1)\in H$. Hence $q \ge j + 2i -1$. On the other hand, since $b \in H$ and $b \equiv 2$ $\mod$ $3$, we get $b \ge 2 \ell = 12 q + 8$. Therefore 
$$
6i + 3j -3q -3 = b - (12 q + 8) \ge 0,
$$
which yields $q \le 2i + j -1$. Thus $q = 2i + j -1$. We then, however, have
$$
c-(b-a) + 1 < a < b \ \  \text{and} \ \ t^{c-(b-a) + 1} \in I=(f,g),
$$
which is impossible. Thus, $b-a \ne 3q+1 + 3j$ for any $1 \le j \le q$, and hence $(a,b)= (\ell, \ell + 3q + 2)$.

$(2)$~We shall prove that $a = 6i$  for some $1 \le i \le q$. Firstly, we assume that $a \equiv 0$ $\mod$ $3$. Hence, $a = 6i~(1 \le i \le 2q-1)$, because $a$ is even. We look at the following table
$$
\begin{tabular}{|c|c|}
\hline
$a$ & $6i$   \\ \hline 
$b-a$ & $3q+1 + 3j~(0 \le j \le q-1)$  \\ \hline 
$b$ &  $3q+1 + 6i + 3j~(0 \le j \le q-1)$ \\ \hline
$2b-a$ &   $6q+2 + 6i + 6j~(0 \le j \le q-1)$ \\ \hline
$c-(b-a)$  &  $9q-1 - 3j~(0 \le j \le q-1)$ \\ \hline
\end{tabular} \vspace{0.3em}
$$
and notice that 
$$
b - (c-(b-a) + 2) = b - (9q+1 - 3j) = 6 (i+j - q).
$$
If $i > q$, then $b > 9q+1 - 3j$. Since $t^{c-(b-a) + 2} \in I$, we get $9q+1 - 3j - 6i \in H$. Therefore
$$
9q+1 - 3j - 6i \ge \ell = 6q + 1,
$$
which yields $q \ge j + 2i \ge 2i > 2q$. This is absurd. Hence, $i \le q$. Because $2b-a \in H$ and $2b-a \equiv 2$ $\mod$ $3$, we have 
$$
2b-a \ge 2 \ell = 12 q + 2.
$$
Therefore, since $2b-a = 6q+2 + 6i + 6j$, we get $6(i+j - q) \ge 0$. In particular, $i+j \ge q$. Suppose now that $i+j > q$. Then, since $b > c-(b-a) + 2$ and $t^{c-(b-a) + 2} \in I$, we have
$$
c-(b-a) + 2 - a \in H \ \ \text{and} \ \ c-(b-a) + 2 - a \equiv 1 \ \mod \ 3,
$$
so that $c-(b-a) + 2 - a \ge \ell = 6q + 1$. Hence $q \ge 2i + j$. Thus $i+j > q  \ge 2i + j$. This is absurd. Consequently, $i+j = q$. Therefore, $(a,b)  = (6i, \ell + 3i)$, and $1 \le i \le q$.

We must show that $a \not\equiv 1~\mod~3$. Assume that $a \equiv 1$ $\mod$ $3$, that is $a = \ell + 3 + 6i$ with $0 \le i \le q-1$. We then have the following table.
$$
\begin{tabular}{|c|c|c|}
\hline
$a$ &  $\ell + 3+ 6i$ \\ \hline 
$b-a$ &  $3q+1 + 3j~(0 \le j \le q-1)$  \\ \hline 
$b$ &   $9q+5 + 6i + 3j~(0 \le j \le q-1)$ \\ \hline
$2b-a$ &   $12q+6 + 6i + 6j~(0 \le j \le q-1)$ \\ \hline
$c-(b-a)$ &   $9q-1 - 3j~(0 \le j \le q-1)$ \\ \hline
\end{tabular} \vspace{0.3em}
$$
Since 
$$
b- (c-(b-a) + 2) = 6(i+j) + 4 >0,
$$
we have $(c-(b-a) + 2) - a = 3(q-j-2i -1)\in H$. Therefore, $q \ge j + 2i +1$. 
Because $b \in H$ and $b \equiv 2$ $\mod$ $3$, we furthermore have $b \ge 2 \ell = 12 q + 2$, whence 
$$
6i + 3j -3q +3 = b - (12 q + 2) \ge 0,
$$
which yields $q \le 2i + j +1$. Thus, $q = 2i + j +1$, and we get
$$
c-(b-a) + 1 < a < b \ \  \text{and} \ \ t^{c-(b-a) + 1} \in I
$$
which is impossible. Consequently, $a \not\equiv 1$ $\mod$ $3$.

$(3)$~Suppose that $a \equiv 0$ $\mod$ $3$. Then $a = 6i~(1 \le i \le 2q+1)$. We consider the table below.
$$
\begin{tabular}{|c|c|}
\hline
$a$ & $6i$   \\ \hline 
$b-a$ & $3q+ 2 + 3j~(1 \le j \le q)$  \\ \hline 
$b$ &  $3q+2 + 6i + 3j~(1 \le j \le q)$ \\ \hline
$2b-a$ &   $6q+4 + 6i + 6j~(1 \le j \le q)$ \\ \hline
$\rmc(H)-(b-a)$  &  $9q+6 - 3j~(1 \le j \le q)$ \\ \hline
\end{tabular} \vspace{0.3em}
$$
Notice that
$$
b - (c-(b-a) + 2) = b - (9q+8 - 3j) = 6 (i+j - q -1). 
$$
If $i > q$, then $b > 9q+8 - 3j$. Since $t^{c-(b-a) + 2} \in I$, we have $9q+8 - 3j - 6i \in H$. Hence
$$
9q+8 - 3j - 6i \ge \ell = 6q + 5,
$$
which yields $q \ge j + 2i -1 \ge 2i > 2q$. This is absurd. Hence, $i \le q$.

Because $2b-a \in H$ and $2b-a \equiv 1$ $\mod$ $3$, we have 
$$
2b-a \ge 2 \ell = 12 q + 10.
$$
Therefore, since $2b-a = 6q+4 + 6i + 6j$, we get $6(i+j - q -1) \ge 0$. In particular, $i+j \ge q+1$. If $i+j > q+1$, then since $b > c-(b-a) + 2$ and $t^{c-(b-a) + 2} \in I$, we have
$$
c-(b-a) + 2 - a \in H \ \ \text{and} \ \ c-(b-a) + 2 - a \equiv 2 \ \mod \ 3,
$$
whence $c-(b-a) + 2 - a \ge \ell = 6q + 5$. Therefore, $q+1 \ge 2i + j$, so that
$$
i+j > q + 1  \ge 2i + j
$$
which is absurd. Consequently, we get $i+j = q+1$, and therefore $(a,b)=(6i,\ell + 3i)$ with $1 \le i \le q$. 

We shall show that $a \not\equiv 2$ $\mod$ $3$. Assume the contrary. We then have $a = \ell + 3 + 6i$ with $0 \le i \le q-1$. We look at the following table.
$$
\begin{tabular}{|c|c|c|}
\hline
$a$ &  $\ell + 3+ 6i$ \\ \hline 
$b-a$ &  $3q+2 + 3j~(1 \le j \le q)$  \\ \hline 
$b$ &   $9q+10 + 6i + 3j~(1 \le j \le q)$ \\ \hline
$2b-a$ &   $12q+12 + 6i + 6j~(1 \le j \le q)$ \\ \hline
$\rmc(H)-(b-a)$ &   $9q+6 - 3j~(1 \le j \le q)$ \\ \hline
\end{tabular} \vspace{0.3em}
$$
Because
$$
b- (c-(b-a) + 2) = 6(i+j) + 2 >0,
$$
we have $(c-(b-a) + 2) - a = 3(q-j-2i)\in H$, whence $q \ge j + 2i$. On the other hand, since $b \in H$ and $b \equiv 1$ $\mod$ $3$, we get $b \ge 2 \ell = 12 q + 10$. Consequently
$$
6i + 3j -3q +3 = b - (12 q + 10) \ge 0,
$$
which yields $q \le 2i + j$. Thus, $q = 2i + j $, so that 
$$
c-(b-a)  < a < b \ \  \text{and} \ \ t^{c-(b-a) } \in I
$$
which is impossible. Thus, $a \not\equiv 2$ $\mod$ $3$.

$(4)$~First let us consider the case where $a \equiv 0$ $\mod$ $3$. Hence, $a = 6i ~(1 \le i \le 2q)$. Look at the following table.
$$
\begin{tabular}{|c|c|c|}
\hline
$a$ & $6i$  & $6i$ \\ \hline 
$b-a$ & $3q+1$  & $3q+2 + 3j~(0 \le j \le q-1)$  \\ \hline 
$b$ & $3q+1+ 6i$ &  $3q+2 + 6i + 3j~(0 \le j \le q-1)$ \\ \hline
$2b-a$ & $6q+2 + 6i$ &  $6q+4 + 6i + 6j~(0 \le j \le q-1)$ \\ \hline
$c-(b-a)$ & $9q + 1$ &  $9q - 3j~(0 \le j \le q-1)$ \\ \hline
\end{tabular} \vspace{0.3em}
$$
Suppose that $b-a = 3q+1$. Then, since
$$
b- (c-(b-a) + 1) = b- (9q + 2) = 6(i-q) - 1,
$$
we get $i \le q$. Indeed, if $i > q$, then $9q + 2-a \in H$, which implies 
$$
9q + 2- a \ge  \ell =6q + 2,
$$
so that $q \ge 2i \ge 2(q+1)$. This is absurd. Hence, $i \le q$. Consequently
$$
b = 3q+1 + 6i \le 9q + 1 < 12 q +1 = c-1,
$$
which yields $b \not\in H$ because $b \equiv 1$ $\mod$ $3$. This is, of course, impossible. Hence, $b-a \ne 3q+1$.

Therefore, $b-a =3q+2 + 3j$ for some $0 \le j \le q-1$, whence
$$
b - (c-(b-a) + 2) = b - (9q+2 - 3j) = 6 (i+j - q).
$$
If $i > q$, then $b > 9q+2 - 3j$. On the other hand, since $t^{c-(b-a) + 2} \in I$, we get $9q+2 - 3j - 6i \in H$. Consequently
$$
9q+2 - 3j - 6i \ge \ell = 6q + 2,
$$
which yields $q \ge j + 2i \ge 2i > 2q$. This is absurd. Thus, $i \le q$.

Since $2b-a \in H$ and $2b-a \equiv 1$ $\mod$ $3$, we have 
$$
2b-a \ge 2 \ell = 12 q + 4.
$$
Therefore, because $2b-a = 6q+4 + 6i + 6j$, we get $6(i+j - q) \ge 0$, whence  $i+j \ge q$. Suppose now that $i+j > q$. Then, since $b > c-(b-a) + 2$ and $t^{c-(b-a) + 2} \in I$, we have
$$
c-(b-a) + 2 - a \in H \ \ \text{and} \ \ c-(b-a) + 2 - a \equiv 2 \ \mod \ 3,
$$
which implies 
$$
c-(b-a) + 2 - a \ge \ell = 6q + 2.
$$
Thus $q \ge 2i + j$, whence $i+j > q \ge 2i + j$. This is absurd. Consequently, $i+j = q$. Hence,  $(a,b)= (6i, \ell + 3i)$ with $1 \le i \le q$.

Lastly, we consider the case where $a \equiv 2$ $\mod$ $3$, that is  $a = \ell + 6i$ for some $0 \le i \le q-1$. This case, we have the following.
$$
\begin{tabular}{|c|c|c|}
\hline
$a$ & $\ell + 6i$  & $\ell + 6i$ \\ \hline 
$b-a$ & $3q+1$  & $3q+2 + 3j~(0 \le j \le q-1)$  \\ \hline 
$b$ & $9q+3 + 6i$ &  $9q+4 + 6i + 3j~(0 \le j \le q-1)$ \\ \hline
$2b-a$ & $12q+4 + 6i$ &  $12q+6 + 6i + 6j~(0 \le j \le q-1)$ \\ \hline
$\rmc(H)-(b-a)$ & $9q + 1$ &  $9q - 3j~(0 \le j \le q-1)$ \\ \hline
\end{tabular} \vspace{0.3em}
$$
Suppose that $b-a = 3q+1$. Then 
$$
b- (c-(b-a) + 2) = 6i \in H.
$$ 
If $i > 0$, then $(c-(b-a) + 2) - a \in H$ and $(c-(b-a) + 2) - a < 3q + 1$. This is impossible, because $(c-(b-a) + 2) - a \equiv 2$ $\mod$ $3$. Therefore, $i=0$ and $(a,b)=(\ell, \ell + 3q + 1)$.

Let us now assume $b-a=3q+2 + 3j~(0 \le j \le q-1)$ and seek a contradiction. Since 
$$
b- (c-(b-a) + 2) = 6(i+j) +2 >0,
$$
we get $(c-(b-a) + 2)  - a = 3(q-j-2i)\in H$. Hence, $q \ge j + 2i$. 

On the other hand, because $b \in H$ and $b \equiv 1$ $\mod$ $3$, we get $b \ge 2 \ell = 12 q + 4$. Therefore 
$$
6i + 3j -3q = b - (12 q + 4) \ge 0,
$$
which yields $q \le 2i + j $. Thus, $q = 2i + j$, and therefore
$$
c-(b-a)  < a < b \ \  \text{and} \ \ t^{c-(b-a) } \in I,
$$
which is a required contradiction. Consequently, $b-a \ne 3q+2 + 3j$ for any $0 \le j \le q-1$. This completes the proof Theorem \ref{3.6}.
\end{proof}

We now have the following, which guarantees that every Ulrich ideal $I$ of $A$ has one of the forms stated in Theorem \ref{e=3} below.

\begin{cor}\label{3.7}
\begin{enumerate}[$(1)$]
\item Suppose that $\ell = 3n + 1$ where $n \ge 3$ is odd. Let $q = \frac{n-1}{2}$.
\begin{enumerate}
\item[$(\rm i)$] If $(a, b) = (\ell, \ell + 3q + 2)$, then 
$$
I = (t^{\ell} + \alpha_1 t^{\ell + 2} + \alpha_2 t^{\ell + 5} + \cdots + \alpha_{q}t^{\ell + 3q -1}, t^{\ell + 3q+2}) 
$$
for some $\alpha_1, \alpha_2, \dots, \alpha_q \in k$.
\item[$(\rm ii)$] If $(a, b) = (6i, \ell + 3i)$ with $1 \le i \le q$, then
$$
I=(t^{6i}+ \alpha_0 t^{\ell} + \alpha_1 t^{\ell+3} + \cdots + \alpha_{i-1}t^{\ell + 3(i-1)}, t^{\ell + 3i})
$$
for some $\alpha_0, \alpha_1, \ldots, \alpha_{i-1} \in k$ such that $\alpha_0 \ne 0$.

\end{enumerate}
\item Suppose that $\ell = 3n + 1$ where $n \ge 2$ is even. Let $q = \frac{n}{2}$. If $(a, b) = (6i, \ell + 3i)$ with $1 \le i \le q$, then 
$$
I=(t^{6i}+ \alpha_0 t^{\ell} + \alpha_1 t^{\ell+3} + \cdots + \alpha_{i-1}t^{\ell + 3(i-1)}, t^{\ell + 3i})
$$
for some  $\alpha_0, \alpha_1, \ldots, \alpha_{i-1} \in k$ such that $\alpha_0 \ne 0$.
\item Suppose that $\ell = 3n + 2$ where $n \ge 3$ is odd. Let $q = \frac{n-1}{2}$. If $(a, b) =  (6i, \ell + 3i)$ with $1 \le i \le q$, then
$$
I = (t^{6i}+ \alpha_0 t^{\ell} + \alpha_1 t^{\ell+3} + \cdots + \alpha_{i-1}t^{\ell + 3(i-1)}, t^{\ell + 3i})
$$
for some  $\alpha_0, \alpha_1, \ldots, \alpha_{i-1} \in k$ such $\alpha_0 \ne 0$.

\item Suppose that $\ell = 3n + 2$ where $n \ge 2$ is even. Let $q = \frac{n}{2}$.
\begin{enumerate}
\item[$(\rm i)$] If $(a, b) = (\ell, \ell + 3q + 1)$, then
$$
I=(t^{\ell} + \alpha_1 t^{\ell + 1} + \alpha_2 t^{\ell + 4} + \cdots + \alpha_{q}t^{\ell + 3q -2}, t^{\ell + 3q+1}) 
$$
for some  $\alpha_1, \alpha_2, \dots, \alpha_q \in k$. 
\item[$(\rm ii)$] If $(a, b) = (6i, \ell + 3i)$ with $1 \le i \le q$, then
$$
I=(t^{6i}+ \alpha_0 t^{\ell} + \alpha_1 t^{\ell+3} + \cdots + \alpha_{i-1}t^{\ell + 3(i-1)}, t^{\ell + 3i}) 
$$
for some  $\alpha_0, \alpha_1, \ldots, \alpha_{i-1} \in k$ such that  $\alpha_0 \ne 0$.
\end{enumerate}
\end{enumerate}
\end{cor}

\begin{proof} 
Let $I \in \calX_A$. Then $\mu_A(I)=2$ and $\fkc \subseteq I$ by Lemma \ref{2.3}. Choosing $f,g \in I$ so that all the conditions stated in Theorem \ref{2.8} are satisfied, we get, similarly as in the proof of Example  \ref{2.9}, the ideal $I$ possesses a minimal system of generators of the specified form. For example, suppose that $\ell = 3n + 1$ where $n \ge 3$ is odd and set $q = \frac{n-1}{2}$. Let us  consider the case where $(a, b) = (\ell, \ell + 3q + 2)$. We firstly write
$$
f = t^{\ell} + \gamma_1t^{\ell+1} +\gamma_2t^{\ell + 2} + \cdots + \gamma_{\ell-3}t^{2\ell -3} + \rho \ \ \text{and}$$
$$g = t^{\ell+3q+ 2} + \delta_1t^{\ell+3q+3} +\delta_2t^{\ell + 3q+4} + \cdots + \delta_{\ell-3q-5}t^{2\ell -3} + \eta,
$$
where 

$\gamma_i \in k$ for all $1 \le i \le \ell-3$ and $\gamma_j = 0$ if $j \equiv 1~\mod~3$,

$\delta_i \in k$ for all $1 \le i \le \ell - 3q - 5$ and $\delta_j = 0$ if $j \equiv 2~\mod~3$, and

$\rho, \eta \in \fkc=t^{2\ell -2}V$. 

\noindent
Then, because $\fkc \subseteq I$, replacing $g$ with $g - \delta_1t^{3q+3}f$, we may assume that $\delta_1=0$, and replacing $g$ with $g-\delta_4t^{3q+6}f$, we may assume that $\delta_1=\delta_4=0$. Repeating this procedure, we may safely assume that $g = t^{\ell+3q+2} + \eta$. As for $f$, replace $f$ with $f - \gamma_3t^3f=(1-\gamma_3t^3)f$, and we may assume that $\gamma_3=0$. Replacing $f$ with $f - \gamma_{6}t^{6}f$, we may also assume that $\gamma_3= \gamma_6=0$. Continuing this procedure, we may now assume that $f$ has the form 
$$
f= t^{\ell} + \alpha_1 t^{\ell + 2} + \alpha_2 t^{\ell + 5} + \cdots + \alpha_{q}t^{\ell + 3q -1} + \rho
$$
with $\alpha_i \in k$ ($1 \le i \le q$). Therefore, since $\fkc=(t^{2\ell-2}, t^{2\ell -1},t^{2\ell}) \subseteq I=(f,g)$, we obtain  $$I = (t^{\ell} + \alpha_1 t^{\ell + 2} + \alpha_2 t^{\ell + 5} + \cdots + \alpha_{q}t^{\ell + 3q -1}, t^{\ell+3q+ 2}, t^{2\ell-2}, t^{2\ell -1},t^{2\ell}),$$
so that  
$$
I = (t^{\ell} + \alpha_1 t^{\ell + 2} + \alpha_2 t^{\ell + 5} + \cdots + \alpha_{q}t^{\ell + 3q -1}, t^{\ell+3q+ 2}),
$$
because $\mu_A(I)=2$ and $t^{\ell + 3q +2} \not\in (t^{\ell} + \alpha_1 t^{\ell + 2} + \alpha_2 t^{\ell + 5} + \cdots + \alpha_{q}t^{\ell + 3q -1}, t^{2\ell-2}, t^{2\ell -1},t^{2\ell})$.

To see the additional condition $\alpha_0 \ne 0$ in $(1)$$(\rm ii)$, $(2)$, $(3)$, and $(4)$$(\rm ii)$, notice that $t^{2\ell} \in I$, and write
$
t^{2\ell} = f \varphi + g \psi
$
with $\varphi, \psi \in A$, and comparing the orders of both sides, we will meet a contradiction, once $\alpha_0=0$. 
\end{proof}

Thanks to Corollary \ref{3.7}, we are now ready to give the main result of this section.

\begin{thm}\label{e=3}
Let $\ell \ge 7$ be an integer such that $\GCD(3, \ell) = 1$ and set $A = k[[t^3, t^{\ell}]]$. Then the following assertions hold true. 
\begin{enumerate}[$(1)$]
\item Suppose that $\ell = 3n + 1$ where $n \ge 3$ is odd. Let $q = \frac{n-1}{2}$. Then 
{\small
\begin{eqnarray*}
\hspace{1.5em}
\calX_A \hspace{-1.3em}
&&=\left\{
(t^{\ell} + \alpha_1 t^{\ell + 2} + \alpha_2 t^{\ell + 5} + \cdots + \alpha_{q}t^{\ell + 3q -1}, t^{\ell + 3q+2}) \mid \alpha_1, \alpha_2, \dots, \alpha_q \in k
\right\} \\ 
&&\bigcup
\left\{(t^{6i}+ \alpha_0 t^{\ell} + \alpha_1 t^{\ell+3} + \cdots + \alpha_{i-1}t^{\ell + 3(i-1)}, t^{\ell + 3i}) \mid 1 \le i \le q, \alpha_0, \ldots, \alpha_{i-1} \in k,  \alpha_0 \ne 0\right\}.
\end{eqnarray*}
}
\item Suppose that $\ell = 3n + 1$ where $n \ge 2$ is even. Let $q = \frac{n}{2}$. Then
{\small
$$
\hspace{1.5em}
\calX_A =\left\{(t^{6i}+ \alpha_0 t^{\ell} + \alpha_1 t^{\ell+3} + \cdots + \alpha_{i-1}t^{\ell + 3(i-1)}, t^{\ell + 3i}) \mid 1 \le i \le q, \alpha_0, \ldots, \alpha_{i-1} \in k,  \alpha_0 \ne 0\right\}.
$$
}
\item Suppose that $\ell = 3n + 2$ where  $n \ge 1$ is odd. Let $q = \frac{n-1}{2}$. Then 
{\small
$$
\hspace{1.5em}
\calX_A =\left\{(t^{6i}+ \alpha_0 t^{\ell} + \alpha_1 t^{\ell+3} + \cdots + \alpha_{i-1}t^{\ell + 3(i-1)}, t^{\ell + 3i}) \mid 1 \le i \le q, \alpha_0, \ldots, \alpha_{i-1} \in k,  \alpha_0 \ne 0\right\}.
$$
}
\item Suppose that $\ell = 3n + 2$ where  $n \ge 2$ is even. Let $q = \frac{n}{2}$. Then
{\small
\begin{eqnarray*}
\hspace{1.5em}
\calX_A \hspace{-1.3em}
&&=\left\{
(t^{\ell} + \alpha_1 t^{\ell + 1} + \alpha_2 t^{\ell + 4} + \cdots + \alpha_{q}t^{\ell + 3q -2}, t^{\ell + 3q+1}) \mid \alpha_1, \alpha_2, \dots, \alpha_q \in k
\right\} \\ 
&&\bigcup
\left\{(t^{6i}+ \alpha_0 t^{\ell} + \alpha_1 t^{\ell+3} + \cdots + \alpha_{i-1}t^{\ell + 3(i-1)}, t^{\ell + 3i}) \mid 1 \le i \le q, \alpha_0, \ldots, \alpha_{i-1} \in k,  \alpha_0 \ne 0\right\}.
\end{eqnarray*}
}
\end{enumerate}
Here, the coefficients $\alpha_i$'s in the given system of generators of each ideal $I \in \calX_A$ are uniquely determined for $I$. 
\end{thm}

\begin{proof}
The proofs of Assertions (1), (2), (3), and (4) are essentially the same. Let us give the proof only for Assertion (1).

It remains to show that the listed ideals are all Ulrich. Firstly, let $f=t^{\ell} + \alpha_1 t^{\ell + 2} + \alpha_2 t^{\ell + 5} + \cdots + \alpha_{q}t^{\ell + 3q -1}$, $g = t^{\ell + 3q+2}$,  and  $I=(f, g)$. We shall show that $I$ is an Ulrich ideal of $A$. We set $a =\ell$, $b =\ell + 3q +2$, and 
$$
v(I) =\{\rmo(h) \mid 0 \ne h \in I\}. 
$$
Since $f,g \in I$, it is standard to check that $c =2a-2, 2a -1, 2a \in v(I)$, 
so that $n \in v(I)$ for all $n \ge c$, whence $\fkc=t^cV \subseteq I$. Therefore, because 
$$
\rmo\left(\frac{g^2}{f}\right) = 2b-a = 2\ell > c,
$$ 
we have $g^2 \in fI$, whence $I^2= fI + (g^2) = fI$. We now consider the exact sequence
$$
0 \to I/(f) \to A/(f) \to A/I \to 0
$$
of $A$-modules, and remember that $\ell_A(A/(f)) = \ell_V(V/fV) = a.$
Let $J = (t^n \mid n \in v(I))$ be the initial ideal of $I$, that is the ideal of $A$ generated by the initial forms of the elements in $I$. We then have $\ell_A(A/J) \le \frac{a}{2}$, counting the number of monomials $t^n \not\in J$. Consequently, since $$\ell_A(A/I) = \ell_A(A/J) \le \frac{a}{2},$$
the epimorphism 
\begin{center}
$\varphi: A/I \to I/(f), \ \ \varphi(1~\mod~I) =g~\mod~(f)$
\end{center} of $A$-modules is an isomorphism. Thus, $I$ is an Ulrich ideal of $A$, which implies the first half of Assertion (1). Similarly, we are able to prove also the second half of Assertion (1).

To see the last claim in Theorem \ref{e=3} for Assertion (1), we must show the following.  
\begin{enumerate}[$(\rm i)$]
\item $(t^{\ell} + \sum_{j=1}^q\alpha_j t^{\ell + 3j-1}, t^{\ell + 3q+2}) = (t^{\ell} + \sum_{j=1}^q\beta_j t^{\ell + 3j-1}, t^{\ell + 3q+2})$, only if $\alpha_j = \beta_j$ for all $1 \le j \le q$.
\item $(t^{6i}+ \sum_{j=0}^{i-1}\alpha_j t^{\ell + 3j}, t^{\ell + 3i}) =(t^{6i}+ \sum_{j=0}^{i-1}\beta_j t^{\ell + 3j}, t^{\ell + 3i})$, only if $\alpha_j = \beta_j$ for all $0 \le j \le i-1$.
\end{enumerate}
Indeed, suppose $(t^{\ell} + \sum_{j=1}^q\alpha_j t^{\ell + 3j-1}, t^{\ell + 3q+2}) = (t^{\ell} + \sum_{j=1}^q\beta_j t^{\ell + 3j-1}, t^{\ell + 3q+2})$ with $\alpha_j, \beta_j \in k$. We write $t^{\ell} + \sum_{j=1}^q\alpha_j t^{\ell + 3j-1} = f\cdot (t^{\ell} + \sum_{j=1}^q\beta_j t^{\ell + 3j-1}) + g \cdot  t^{\ell + 3q+2}$ for some $f, g \in A$. By setting $f = \gamma + f_0 + f_1 + \xi$ where $\gamma \in k$, $f_0 \in \sum_{j=1}^{4q+1}kt^{3j}$, $f_1 \in \sum_{j=0}^{2q}kt^{\ell + 3j}$, and $\xi \in t^{c}V$, we then have the equalities
\begin{eqnarray*}
t^{\ell} + \sum_{j=1}^q\alpha_j t^{\ell + 3j-1} &=& f\cdot \left(t^{\ell} + \sum_{j=1}^q\beta_j t^{\ell + 3j-1}\right) + g \cdot  t^{\ell + 3q+2} \\
&=& (\gamma + f_0)\cdot \left(t^{\ell} + \sum_{j=1}^q\beta_j t^{\ell + 3j-1}\right)\\ 
&{ }& + \text{(terms of degree greater than $\ell + 3q -1$)}.
\end{eqnarray*}
Comparing the coefficients of $t^n$ in both sides, we get $\gamma=1$ and $f_0 = 0$. Hence, Assertion (i) follows, that is $\alpha_j = \beta_j$ for all $1 \le j \le q$. We similarly have, assuming $(t^{6i}+ \sum_{j=0}^{i-1}\alpha_j t^{\ell + 3j}, t^{\ell + 3i}) =(t^{6i}+ \sum_{j=0}^{i-1}\beta_j t^{\ell + 3j}, t^{\ell + 3i})$ with $\alpha_j, \beta_j \in k$, that $\alpha_j = \beta_j$ for all $0 \le j \le i-1$, which completes the proof of Theorem \ref{e=3}.
\end{proof}

Let us note some direct consequences.

\begin{cor}\label{e=3cor}
Let $H = \left<3,\ell\right>$ where $\ell \ge 7$ is an integer such that $\operatorname{GCD}(3,\ell)=1$ and let $A = k[[H]]$ stand for the semigroup ring of $H$ over a field $k$. Then the following assertions hold true.
\begin{enumerate}[{\rm $(1)$}]
\item $\calX_A \ne\emptyset$.
\item $\#(\calX_A)< \infty$ if and only if $\#(k) < \infty$.
\item The ring $A$ contains no Ulrich ideals generated by monomials in $t$ if and only if $\ell = 3n + 1$ or $\ell = 3n + 2$ for some even integer $n \ge 2$. 
\end{enumerate}
\end{cor}

For example, consider the simplest case $A =k[[t^3, t^7]]$. Then $$\calX_A=\{(t^6 + \alpha t^7, t^{10}) \mid 0 \ne \alpha \in k\}.$$ Therefore, $\#(\calX_A)= \#(k)-1$, and $A$ contains no Ulrich ideals generated by monomials in $t$.


\section{The case where $H = \left<4,13\right>$}
Example \ref{2.9} and Theorem \ref{e=3} are  the starting points for our study of $\calX_{k[[H]]}$ where $H$ is a numerical semigroup of small multiplicity. Naturally, the next target should be the case of multiplicity $4$. Nevertheless, contrary to our lighthearted expectations, even for $A = k[[t^4,t^{13}]]$ (which is one of the simplest cases) the task of determining the elements of $\calX_A$ is much more formidable than that of the case of $e=2$ or $3$, as we shall report in this section.

Let $H = \left<4, 13\right>$
$$
 \ytableausetup{mathmode, boxsize=2em}
 \begin{ytableau}
 *(lightgray) 0 & 1 &  2 & 3 \\
 *(lightgray) 4 & 5 & 6 & 7  \\
 *(lightgray) 8   &  9 & 10 & 11\\ 
 *(lightgray) 12   &  *(lightgray) 13 & 14 & 15\\ 
 *(lightgray) 16   &   *(lightgray)17 & 18 & 19\\
 *(lightgray) 20   &   *(lightgray)21 & 22 & 23\\ 
 *(lightgray) 24   &   *(lightgray)25 &  *(lightgray)26 & 27\\ 
 *(lightgray) 28   &  *(lightgray) 29 &  *(lightgray)30 & 31\\
 *(lightgray) 32   &  *(lightgray) 33 &  *(lightgray)34 & 35\\
 *(lightgray) 36   &  *(lightgray) 37 &  *(lightgray)38 & *(lightgray) 39\\
 *(lightgray) \vdots   &  *(lightgray) \vdots &  *(lightgray) \vdots &  *(lightgray) \vdots
 \end{ytableau} 
 $$
and let $A = k[[t^4,t^{13}]]$ denote the semigroup ring of $H$ over a field $k$.

Our goal is the following.

\begin{thm}\label{e=4} We have
\begin{eqnarray*}
\calX_{k[[t^4,t^{13}]]} &=& \{(t^{12} + 2\beta t^{17} + \alpha t^{26}, t^{21} + \beta t^{26}) \mid \alpha, \beta \in k, \ \beta \ne 0\} \\
&\bigcup& \{(t^{16} + 2\beta t^{17} + \alpha_2 t^{21} + \alpha_3t^{26}, t^{25} + \beta t^{26}) \mid \alpha_2, \alpha_3, \beta \in k, \ \beta \ne 0\} \\
&\bigcup & \{(t^{4} + \alpha t^{13}, t^{26}) \mid \alpha \in k \} \\
&\bigcup & \left\{(t^{8} + \alpha_1 t^{13} + \alpha_2t^{17}, t^{26}) \mid \alpha_1, \alpha_2 \in k \right\} \\
&\bigcup & \{(t^{12} + \alpha_1 t^{13} + \alpha_2t^{17} + \alpha_3t^{21}, t^{26}) \mid \alpha_1, \alpha_2, \alpha_3 \in k \} \\
&\bigcup & \{(t^{16} + \alpha_1 t^{17} + \alpha_2t^{21} + \alpha_3t^{25}, t^{26}) \mid \alpha_1, \alpha_2, \alpha_3 \in k \} \\
&\bigcup & \{(t^{20} + \alpha_1 t^{21} + \alpha_2t^{25} + \alpha_3t^{29}, t^{26} + \beta t^{29}) \mid \alpha_1, \alpha_2, \alpha_3, \beta \in k, \ \alpha_1^3 = 2\beta \} \\
&\bigcup & \{(t^{24} + \alpha_1 t^{25} + \alpha_2t^{29} + \alpha_3t^{33}, t^{26} + \beta_1t^{29} + \beta_2 t^{33}) \mid \alpha_1, \alpha_2, \alpha_3, \beta_1, \beta_2 \in k, \\ 
&&\ \ \alpha_1 = 0 \ \text{if} \ \ch k = 2;  \ \alpha_1 = \alpha_2 = \beta_1 = \beta_2 = 0 \ \text{if} \ \ch k \ne  2 \}.
\end{eqnarray*}
For each $I \in \calX_{k[[t^4,t^{13}]]}$, the elements of $k$ which appear in the listed expression are uniquely determined by $I$. 
\end{thm}

\noindent
This result shows that if $k$ is a finite field, then $k[[t^4,t^{13}]]$ contains only finitely many Ulrich ideals, but if $k$ is infinite, then it contains numerous Ulrich ideals which are not generated by monomials in $t$.

The proof of Theorem \ref{e=4} is divided into several steps. First of all, let us fix the following.

\begin{setting}\label{setting2}
Let $I$ be an Ulrich ideal of $A$ and choose $f, g \in I$ so that all the conditions stated in Theorem \ref{2.8} are satisfied. Namely 
\begin{itemize}
\item $a, b \in H$, $0 < a < c$, $0 < a < b < a + c$,
\item $b-a \not\in H$, $2b-a \in H$, $a = 2\cdot \ell_A(A/I)$, 
\end{itemize}
where $c = \rmc(H)$, $a = \rmo(f)$, and $b = \rmo(g)$.
We set  
$
B=\bigcup_{n\geq 0}\left[I^n: I^n\right]
$
and let $\m_B$ be the maximal ideal of $B$. Let $\xi = \frac{g}{f} \in B$ and $H_1= \{ \rmo(x) \mid 0 \ne x \in B \}$, the value semigroup of $B$. Hence, $B = k[[t^4, t^{13},\xi]]$. Notice that $B=f^{-1}I$ is a Gorenstein ring, $\mu_A(B) = 2$, and $b-a \in H_1$. In particular, $H_1$ is a symmetric numerical semigroup. 
\end{setting}

Our strategy is the following. Similarly as in Section 3, first of all, we enumerate all the possible semigroups $H_1$. Secondly, we determine the possible pairs $(a,b)$ according to the list of possible $H_1$. Lastly, we will show that the pairs $(a,b)$ actually appear to be the data for some $I \in \calX_A$, pinpointing the elements of $\calX_A$.

We denote by $\mu(H_1)$ the number of a minimal system of generators of $H_1$. Let us begin with the following.

\begin{prop}\label{5.0}
$\mu(H_1)=2$ or $3$.
\end{prop}

\begin{proof}
Notice that $\mu(H_1) \le 4$, since $\rme(H_1) \le 4$. Because $\mu_A(V) = 4$ and $\mu_A(B)=2$, we get $B \ne V$, whence  $1 \not\in H_1$ and $\mu(H_1) \ge 2$. Suppose $\mu(H_1)=4$ and let $C = k[[H_1]]$. Then $C$ is a Gorenstein ring and $\mu_C(\m_C)=4$ (here $\m_C$ denotes the maximal ideal of $C$). Therefore, because $\mu_C(\m_C) \le \rme(C)= \rme(H_1) \le 4$, $C$ has minimal multiplicity, whence $C$ must be a hypersurface of multiplicity $2$, which is impossible.
\end{proof}

\begin{prop}\label{prop0}
If $2 \in H_1$, then $H_1=\left<2, 13\right>$.
\end{prop}

\begin{proof}
Suppose that $2 \in H_1$. Then $H_1 =\left<2, \alpha \right>$ for some odd integer $3 \le \alpha \le 13$. Take $\eta, \rho \in B$ so that $\rmo(\eta) = 2$ and $\rmo(\rho) = \alpha$. Then, $B= k[[\eta, \rho]]$, and therefore, because $\ell_k(B/\m B)=\ell_A(B/\m B) =2$, the elements $1, \eta, \rho$ $\mod~\m B$ are linearly dependent over $k$ inside $B/\m B$. Therefore, $\rho \in \m B$. In fact, choose $\mathbf{0} \ne \left(\begin{smallmatrix}
\alpha\\
\beta\\
\gamma
\end{smallmatrix}
\right)
 \in k^3$ so that $$\alpha + \beta \eta + \gamma \rho \in \m B.$$ We then have $\alpha = 0$. If $\beta \ne 0$, then $\eta + \gamma \rho \in \m B$ for some $\gamma \in k$, whence
$$\eta + \gamma \rho =t^4 \varphi + t^{13} \psi
$$
with $\varphi, \psi \in B$. This is impossible, since $\rmo(\eta + \gamma \rho)=2$. Thus, $\beta = 0$, and $\rho \in \m B$. Let us write $$\rho = t^4 \varphi_1 + t^{13} \psi_1$$ with $\varphi_1, \psi_1 \in B$. If $\alpha < 13$, then $\rmo(\varphi_1) = \alpha-4 \in H_1 =\left<2, \alpha \right>$, which is impossible. Hence, $\alpha = 13$ and $H_1 = \left<2, 13\right>$. 
\end{proof}

\begin{lem}\label{lemma1}
$H_1 \ne \left<3,4\right>, \left<4, 5\right>, \left<4,5,6\right>, \left<4,6,7 \right>, \left<4, 6, 9 \right>, \left<4,9,10\right>$
\end{lem}

\begin{proof}
Assume that $H_1 = \left<3,4\right>$ and write $B=k[[\eta, t^4]]$ with $\eta \in B$ such that $\rmo(\eta)= 3$. Then, since $\eta^2 \in \m B$ by the proof of Lemma \ref{2.1}, we have $\eta^2 = t^4\varphi + t^{13}\psi$ for some $\varphi, \psi \in B$, which shows $2 = \rmo(\varphi) \in H_1=\left<3,4 \right>$. This is impossible.

Assume that $H_1=\left<4, 5\right>$ and write $B= k[[t^4, \eta]]$ with $\eta \in B$ such that $\rmo(\eta)=5$. Then, since $\eta^2 \in \m B$ for the same reason as above, we get $$\eta^2= t^4 \varphi + t^{13}\psi$$ with $\varphi, \psi \in B$, which forces $6=\rmo(\varphi) \in H_1=\left<4, 5\right>$. This is absurd.

Assume that $H_1 =\left<4, 5, 6 \right>$ and write $B= k[[t^4, \eta, \rho]]$ with $\eta, \rho \in B$ such that $\rmo(\eta)=5$ and $\rmo(\rho)=6$. Then, the elements $1, \eta, \rho$ $\mod~\m B$ are linearly dependent over $k$ inside $B/\m B$, and we have $\rho \in \m B$ for the same reason as in the proof of Proposition \ref{prop0}. Writing $$\rho = t^4 \varphi_1 + t^{13} \psi_1$$ with $\varphi_1, \psi_1 \in B$, we see $2=\rmo(\varphi_1) \in H_1=\left<4,5,6\right>$, which is impossible. Hence, $H_1 \ne \left<4, 5, 6 \right>$.

The assertion that $H_1 \ne \left<4,6,7 \right>, \left<4,6,9 \right>, \left<4,9,10\right>$ is similarly proved.
\end{proof}

\begin{lem}\label{lemma2}
If $2 \not\in H_1$, then $3,5 \not\in H_1$.
\end{lem}

\begin{proof}
Assume that $3 \in H_1$. Then, since $H_1 \supsetneq \left<3,4\right>$ by Lemma \ref{lemma1}, we get $H_1 \supseteq \left<3,4,5\right>$ (remember that the $k[[t^3,t^4]]$-submodule $k[[H_1]]/k[[t^3,t^4]]$ of $V/k[[t^3,t^4]]$ contains a unique socle $t^5~\mod~k[[t^3,t^4]]$, since $k[[t^3,t^4]]$ is a Gorenstein ring). Consequently, $H_1 = \left<3,4,5 \right>$, since $\left<3,4,5\right> \supseteq H_1$ (because $1, 2 \not\in H_1$). This is, however, impossible,  since $\left<3,4,5\right> $ is not symmetric.

Assume that $5 \in H_1$. Then, $H_1 \supsetneq \left<4,5\right>$ by Lemma \ref{lemma1}. Hence, for the same reason as above $H_1 \supsetneq  \left<4,5,11\right>$ (notice that $\left<4,5,11\right>$ is not symmetric). Consequently, considering the socle of the $k[[t^4,t^5, t^{11}]]$-module $V/k[[t^4,t^5,t^{11}]]$ which is spanned by the images of $t^6$ and $t^7$, we get $6 \in H_1$ or $7 \in H_1$. Therefore, $H_1 \supsetneq \left<4,5, 6\right>$ by Lemma \ref{lemma1} or $H_1 \supsetneq \left<4,5,7\right>$ (since $\left<4,5,7\right>$ is not symmetric). Suppose now that $7 \in H_1$. Then,  considering the socle of the $k[[t^4,t^5, t^{7}]]$-module $V/k[[t^4,t^5,t^7]]$, we get $3 \in H_1$ or $6 \in H_1$. Hence, $6 \in H_1$ even in the case where $7 \in H_1$, because $3 \not\in H_1$ as is shown above. Therefore, $H_1 \supsetneq \left<4,5,6\right>$, whence $H_1 \supseteq \left<4,5,6,7\right>$, because $\left<4,5,6\right>$ is symmetric and the socle of $k[[t^4,t^5,t^6]]$-module $V/k[[t^4,t^5,t^6]]$ is spanned by the image of $t^7$. Thus, $H_1 = \left<4,5,6,7\right>$ since $\left<4,5,6,7\right> \supseteq H_1$, which is impossible because $\left<4,5,6,7\right>$ is not symmetric. Hence, $5 \not\in H_1$, as is claimed.
\end{proof}

We now give an account of the possible semigroups $H_1$ in the following way. We will later show that all of the listed semigroups appear as the value semigroups $v(A^I)$ of $A^I$ for some $I \in \calX_A$.

\begin{prop}\label{5.1}
\begin{enumerate}[$(1)$]
\item If $\mu(H_1) = 2$, then $H_1 = \left<2, 13\right>$ or $H_1 =\left<4, 9\right>$.
\item If $\mu(H_1) = 3$, then $H_1 = \left<4, 9, 14\right>$ or $H_1 =\left<4, 2n, 13\right>$ for some  $n \in \{3, 5, 7, 9, 11\}$.
\end{enumerate}
\end{prop}

\begin{proof}
Thanks to Proposition \ref{prop0}, we may assume that $2 \not\in H_1$. Then, Lemma \ref{lemma2} shows $\min [H_1 \setminus \{ 0 \}] = 4$. Therefore, if $\mu(H_1)=2$, then $H_1 = \left<4, \alpha \right>$ for some odd integer $\alpha \ge 7$ (see Lemma \ref{lemma1}). Because $13 \in H_1$, we readily get $\alpha = 9$. This proves Assertion (1).

Suppose that $\mu(H_1)=3$. Symmetric numerical semigroups which are minimally generated by three elements are complete intersections (\cite[Theorem 3.10]{Herzog}, \cite[Corollary 10.5]{RG}) and they are obtained by gluing (\cite[Section 3]{Herzog}, \cite[Proposition 3]{W}). According to the structure theorem, our semigroup $H_1$ must have one of the following forms.
\begin{enumerate}
\item[${\rm (i)}$] $H_1 = \left<4, a_2, 2n\right>$, where both $a_2$ and $n$ are odd such that $a_2 \ge 5$ and $n \ge 3$, $a_2 \ne n$, and $a_2 \in \left<2, n \right>$. Hence $n < a_2$.
\item[${\rm (ii)}$] $H_1 = \left<4, da_2', da_3'\right>$, where $d \ge 3$ is odd, $a_2', a_3' \ge 2$ such that $\GCD(a_2', a_3') = 1$, $4 \in \left<a_2', a_3'\right>$, $4 \not\in \{a_2', a_3'\}$, and either $a_2'$ or $a_3'$ is even. 
\end{enumerate}
Firstly, we consider Case (i). Notice that $a_2 \le 13$ since $13 \in H_1$, while $a_2 \ne 5$ by Lemma \ref{lemma2}. If $a_2 < 13$, then $a_2 = 7, 9, 11$. Let us write 
$$13 = 4 \alpha + a_2 \beta + 2n \gamma$$ with $\alpha, \gamma \ge 0$ and $\beta > 0$. Suppose that $a_2=7$; hence $H_1=\left<4,7,2n\right>$. Then, $\beta = 1$ and 
$$
3 =2\alpha + n \gamma.
$$ 
Since $n \ge 3$ and $n$ is an odd integer, we have $\alpha=0$ and $n=3$, whence $H_1=\left<4, 7, 6\right>$, which violates Lemma \ref{lemma1}. Suppose that $a_2 = 9$. Then, since $13 = 4\alpha + 9\beta + 2n \gamma$, we have $\beta = 1$ and $2 =2\alpha + n\gamma$. Therefore, $\gamma = 0$ and $\alpha =1$, while $n =3, 5, 7$ since $9 \in \left<2,n\right>$ and $n < 9$. Suppose that $a_2=11$. Then $13 = 4\alpha + 11\beta +2n \gamma$, so that $\beta = 1$ and $1 = 2\alpha + n \gamma$, which impossible. Consequently, if $a_2 < 13$, then $H_1=\left<4, 9, 14\right>$, since $H_1 \ne \left<4,6,9\right>, \left<4,9,10\right>$ by Lemma \ref{lemma1}. If $a_2 = 13$, then $H_1 =\left<4, 13, 2n\right>$ where $n = 3, 5, 7, 9, 11$, because $n < a_2$ and $n$ is odd.

We now consider Case (ii) and will show that it doesn't occur. Without loss of generality we may assume that $a_2' < a_3'$. We set $a_2 = da_2'$ and $a_3 = da_3'$. Then $a_2 \le 13$, since $13 \in H_1$. Therefore, $a_2 =6,7,9, 10,11,13$, since $\mu(H_1)=3$ and $5 \not\in H_1$. The number $a_2$ cannot be prime; otherwise $d= a_2$ and $a_2$ divides $a_3$. Therefore, $a_2= 6, 9, 10$. Suppose $a_2=6$. Then, $d=3$ and $a_3 = 3q$,  where $q=a_3' \ge 3$ is odd. Hence, $q=3$ because $13 \in H_1$, so that $H_1 =\left<4,6,9\right>$, which violates Lemma \ref{lemma1}. Suppose $a_2=9$. Then $d=3$ and $a_3=3q$, where $q \ge 4$ is even. Hence, because $a_3 < \rmc(\left<4,9\right>) = 24$ (notice that $H_1=\left<4, 9, a_3\right>$ and $\mu(H_1)=3$), we have $q=4$ or $q=6$, so that $H_1=\left<4,9\right>$, which is absurd. Suppose $a_2=10$. Then $d=5$ and $a_3=5q$. Hence, $q \ge 3$ is odd, because $4 \in \left<2,q \right>$ and $\operatorname{GCD}(2,q)=1$. Therefore, $a_3 \ge 15$, whence $13 \not\in \left<4,10, a_3\right>=H_1$. This is absurd. Thus, Case (ii) is excluded and Assertion (2) follows.
\end{proof}

\begin{lem}\label{4.8}
\begin{enumerate}[$(1)$]
\item Suppose that $H_1 = \left<4, 9\right>$ or $H_1 = \left<4, 9, 14\right>$. Then $b-a = 9$. 
\item Suppose that $H_1 =\left<4, 2n, 13\right>$ for some $n \in \{1, 3, 5, 7, 9, 11\}$. Then $b-a = 2n$. 
\end{enumerate}
\end{lem}

\begin{proof}
$(1)$~We choose $\eta \in B$ so that $\rmo(\eta) = 9$. Firstly, suppose that $H_1=\left<4,9\right>$ and write $b-a = 4 \alpha + 9 \beta$ with $\alpha, \beta  \ge 0$. Then $\beta > 0$ since $b-a \not\in H$, whence $b-a \ge 9$. Let us write 
$$\eta = t^4 \varphi + t^{13} \psi + \xi \delta$$ with $\varphi, \psi, \delta \in B$. If $b-a= \rmo(\xi) > 9$, then $\rmo(\varphi) = 5 \in H_1=\left<4,9 \right>$, which is impossible. Hence, $b-a = 9$. Next, suppose that $H_1=\left<4,9, 
14\right>$ and write $$b-a= 4\alpha +9 \beta + 14 \gamma $$
with $\alpha, \beta, \gamma \in B$. We then have $b-a \ge 9$, since $b -a \not\in H$. Because $\eta \in k[[t^4, t^{13}, \xi]]$, writing 
$\eta = t^4 \varphi + t^{13} \psi + \xi \delta$ with $\varphi, \psi, \delta \in B$, similarly as in the case where $H_1=\left<4,9\right>$ we get $b-a \le 9$, whence $b-a=9$ as claimed.

$(2)$~Since $b-a \in H_1=\left<4, 2n, 13\right> \setminus H$, we have $b-a \ge 2n$. Take $\eta \in B$ so that  $\rmo(\eta) = 2n$ and write $\eta = t^4 \varphi + t^{13}\psi + \xi \delta$ with $\varphi, \psi, \delta \in B$. If $b-a=\rmo(\xi) > 2n$, then $\rmo(t^4 \varphi + t^{13}\psi) = 2n$, so that $$B = k[[t^4, t^{13}, t^4 \varphi + t^{13}\psi]]$$ and $\m_B = (t^4, t^{13}, t^4 \varphi + t^{13}\psi) = (t^4, t^{13})B = \m B$. Hence,  $\ell_A(B/\m B) = 1$, so that $A = B$. This is absurd. Thus $b-a = 2n$. 
\end{proof}

We are now ready to determine the pair $(a,b)$. Let $v(I) = \{\rmo(x) \mid 0 \ne x \in I\}$. Hence $v(I) = a + H_1$, because $I = fB$.

\begin{prop}
\begin{enumerate}[$(1)$]
\item If $H_1 = \left<4, 9\right>$, then $(a, b) = (12, 21)$. 
\item If $H_1 = \left<4, 9, 14\right>$, then $(a, b) = (16, 25)$. 
\item Let $n \in \{ 1, 3, 5, 7, 9, 11\}$. If $H_1 =\left<4, 2n, 13\right>$, then $(a, b) = (26-2n, 26)$. 
\end{enumerate}
\end{prop}

\begin{proof}
$(1), (2)$~We have $a \in \{4,   8,  12, 16, 20, 24, 26, 28, 30, 32, 34\}$, since $a \in H$ is even and $a < c=36$. We see $b-a=9$ by Lemma \ref{4.8}, and similarly as in the proof of Example \ref{ex0} we consider the table

\vspace{-0.3em}
$$
\begin{tabular}{|c|c|c|c|c|c|c|c|c|c|c|c|}
\hline
$a$ & $4$  & $8$ & $12$ & $16$ & $20$ & $24$ & $26$ & $28$ & $30$ & $32$ & $34$\\ \hline 
$b-a$ & $9$  & $9$ & $9$ & $9$ & $9$ & $9$ & $9$ & $9$ & $9$ & $9$ & $9$\\ \hline 
$b$ & $13$  & $17$ & $21$ & $25$ & $29$ & $33$ & $35$ & $37$ & $39$ &$41$ & $43$\\ \hline
$2b-a$ & $22$  & $26$ & $30$ & $34$ & $38$ & $42$ & $44$ & $46$ & $48$ & $50$ & $52$ \\ \hline
$36-(b-a)$ & $27$  & $27$ & $27$ & $27$ & $27$ & $27$ & $27$ & $27$ & $27$ & $27$ & $27$\\ \hline
\end{tabular} \vspace{1em}
$$
where the values of each column indicate the possible values of $b, 2b-a$, and $36 -(b-a)$, when $a$ is given. The value $36-(b-a)=27$ and Lemma \ref{2.3} indicate that $t^n \in I=(f,g)$ for every $n \in H$ such that $n \ge 27$.

This table tells us many things. For example, $a \ne 4$, since $2b -a= 22 \not\in H$. We have $a \ne 24$. In fact, if $a = 24$, then $b= 33$, so that $t^{29} \in I= (f,g)$, which is impossible because $5 \not\in H$. Thanks to the same observation, we readily get that $a \ne 26, 28, 30, 32, 34$. Hence, $a \in \{8, 12, 16, 20\}$. We will show that $a \ne 8, 20$. Suppose that $a=8$. We then have $v(I) = 8 + H_1 \ni 35$, since $9 \in H_1$. This is however, impossible, because $v(I) \subseteq H$ and $35 =c-1 \not\in H$. Hence, $a \ne 8$. If $a=20$, then $t^{30} \in I$ since $36-(b-a) < 30$. We consider the expression 
\begin{eqnarray*}
t^{30}&=& (t^{20}+f_{21}t^{21}+ \cdots)(\varphi_0+\varphi_4t^4 + \varphi_8t^8+\varphi_{12}t^{12} + \varphi_{13}t^{13}+\cdots)\\
&+&(t^{29}+g_{30}t^{30}+\cdots)(\psi_0+\psi_4t^4 + \psi_8t^8+\psi_{12}t^{12} + \cdots)
\end{eqnarray*}
with coefficients $f_{i}, g_{i}, \varphi_{i}, \psi_{i} \in k$ for each $i \in H$. Then, comparing the order of both sides, it is straightforward to check that $\varphi_0,\varphi_4,\varphi_8=0$ and $\psi_0=0$, so that the term $t^{30}$ doesn't appear in the right hand side. Hence, $a \ne 20$.

Let $a = 12$. If $H_1 = \left<4, 9, 14\right>$, then $H_1 \ni 23$, so that $v(I)=12 +H_1 \ni 35$. This is impossible. Therefore, $H_1 \ne \left<4, 9, 14\right>$ and hence $H_1 = \left<4,9\right>$, if $a = 12$.

Let $a=16$ and we will show $H_1=\left<4, 9, 14 \right>$. Assume that $H_1 \ne \left<4, 9, 14 \right>$. We then have $H_1=\left<4,9 \right>$, whence $v(I) \ni 16+18=34$, so that $t^{34} \in I$ because $35 \not\in H$ and $t^{36}V \subseteq I$. Therefore, the $k$-space $A/I$ is spanned by the images of the monomials $\{t^n\}_{n \in H, \ 0 \le n \le 33}$. Because among the images of these monomials, there are relations induced from the vanishing $$f= t^{16} + (\text{higher~terms}) \equiv 0~\mod~I\ \ \text{and}\ \  g=t^{25}+(\text{higher~terms})\equiv 0 ~\mod~I,$$
the $k$-space $A/I$ is spanned by the images of the following $9$ monomials
$$
1, t^4, t^8, t^{12}, t^{13}, t^{17}, t^{21}, t^{26}, \ \text{and} \ t^{30}.
$$
Of course, these nine monomials cannot be linearly independent over $k$, since $$\ell_A(A/I) = \frac{a}{2}=8$$ (see Theorem \ref{2.8} (3)).  Therefore, there must be a non-trivial relation, say
$$
a_0 + a_4t^4 + a_8t^8 + a_{12}t^{12} + a_{13}t^{13} + a_{17}t^{17} + a_{21}t^{21} + a_{26}t^{26} + a_{30}t^{30} \in I
$$
with $a_i \in k$. Nevertheless, because $v(I) = 16 +  \left<4, 9\right>$, we readily get $a_0=a_4=a_8=a_{12}=a_{13}=0$, so that 
$$
a_{17}t^{17} + a_{21}t^{21} + a_{26}t^{26} + a_{30}t^{30} \in I.
$$
Hence, $a_{17}=a_{21}=a_{26}=a_{26}=a_{30}=0$ also, because $17,21,26,30 \not\in 16+\left<4,9 \right>$, which violates the non-triviality of the relation. Thus, if $a=16$, then $H_1 \ne \left<4,9 \right>$, so that $H_1=\left<4,9,14 \right>$. This completes the proof of Assertions (1) and (2).

$(3)$~We have $a \in \{4,   8,  12, 16, 20, 24, 26, 28, 30, 32, 34\}$ and $b-a \in \{2,6,10,14,18,22\}$. Similarly as above, we consider the tables where $b-a$ is fixed and $a$ takes various values. Our aim is to show that only the cases $(a,b)=(4\ell, 26)$ $(1 \le \ell \le 6)$ are possible.

Let us examine the case where $b-a=2$ or $6$. Suppose that $b-a=2$ and we have the following.

\vspace{-0.3em}
$$
\begin{tabular}{|c|c|c|c|c|c|c|c|c|c|c|c|}
\hline
$a$ & $4$  & $8$ & $12$ & $16$ & $20$ & $24$ & $26$ & $28$ & $30$ & $32$ & $34$\\ \hline 
$b-a$ & $2$  & $2$ & $2$ & $2$ & $2$ & $2$ & $2$ & $2$ & $2$ & $2$ & $2$\\ \hline 
$b$ & $6$  & $10$ & $14$ & $18$ & $22$ & $26$ & $28$ & $30$ & $32$ &$34$ & $36$\\ \hline
$2b-a$ & ${ \times}$  & ${ \times } $ & ${ \times } $ & ${  \times} $ & ${ \times } $ & $28$ & $30$ & $32$ & $34$ & $36$ & $38$ \\ \hline
$36-(b-a)$ & ${ \times } $  & ${  \times} $ & ${ \times } $ & ${ \times } $ & ${ \times } $ & $34$ & $34$ & $34$ & $34$ & $34$ & $34$\\ \hline
\end{tabular} \vspace{1em}
$$
Therefore, because $6, 10,14,18,22 \not\in H$, $a \ne 4, 8, 12,16, 20$. If $a = 28, 30, 32, 34$, then, since $37 \in H$ and $34 < 37$, we have $t^{37} \in I=(f,g)$. This is, however,  impossible, which we can check similarly as in the case where $a=20$ of Assertions (1) and (2), writing $t^{37} = f\varphi + g \psi$ with $\varphi, \psi \in A$. Hence, $(a,b)=(24, 26)$, if $b-a=2$.

 Suppose that $b-a=6$ and we have the following. 
\vspace{-0.3em}
$$
\begin{tabular}{|c|c|c|c|c|c|c|c|c|c|c|c|}
\hline
$a$ & $4$  & $8$ & $12$ & $16$ & $20$ & $24$ & $26$ & $28$ & $30$ & $32$ & $34$\\ \hline 
$b-a$ & $6$  & $6$ & $6$ & $6$ & $6$ & $6$ & $6$ & $6$ & $6$ & $6$ & $6$\\ \hline 
$b$ & $10$  & $14$ & $18$ & $22$ & $26$ & $30$ & $32$ & $34$ & $36$ &$38$ & $40$\\ \hline
$2b-a$ & ${ \times}$  & ${\times } $ & ${\times } $ & ${\times } $ & ${ 32} $ & $36$ & $38$ & $40$ & $42$ & $44$ & $46$ \\ \hline
$36-(b-a)$ & ${ \times} $  & ${\times} $ & ${\times } $ & ${\times} $ & ${30 } $ & $30$ & $30$ & $30$ & $30$ & $30$ & $30$\\ \hline
\end{tabular} \vspace{1em}
$$
Hence, $a \ne 4, 8, 12,16$. Suppose $a \ge 24$. Then, since $33 \in H$ and $33 > 30$, we have $t^{33} \in I=(f,g)$. This is impossible, which we can show, writing $t^{37} = f\varphi + g \psi$ with $\varphi, \psi \in A$ and just counting $\rmo(\varphi)$ and $\rmo(\psi)$. Hence, $(a,b)=(20, 26)$, if $b-a=6$. 

The proofs of the other cases are quite similar to that of the case where $b-a=6$, which we would like to leave to the reader.
\end{proof}

We are now in a position to describe the normal form of systems of generators for a given $I \in \calX_A$. First we consider the case where $H_1 =  \left<4, 9 \right>$ or $H_1= \left<4,9, 14 \right>$.

\begin{thm}\label{b-a=9}
\begin{enumerate}[$(1)$]
\item Suppose that $H_1 = \left<4, 9\right>$. Then 
$$
I =(t^{12} + 2\beta t^{17} + \alpha t^{26}, t^{21} +\beta t^{26})
$$
where $\alpha, \beta \in k$ and $\beta \ne 0$.
\item Suppose that $H_1 = \left<4, 9, 14\right>$. Then 
$$
I = (t^{16} + 2\beta t^{17} + \alpha_2 t^{21} + \alpha_3t^{26}, t^{25} + \beta t^{26}) 
$$
where $\alpha_2, \alpha_3, \beta \in k$ and $\beta  \ne 0$.
\end{enumerate}
\end{thm}

\begin{proof}
Since $c-(b-a) =36-9= 27$, by Lemma \ref{2.3} we have $t^q \in I$ for all $q \in H$ such that $q \ge 27$. 

(1)~Therefore, for the same reason as in the proof of Corollary \ref{3.7} we may assume that
$$
f = t^{12} + \alpha_1 t^{13} + \alpha_2 t^{17} + \alpha_3t^{26} \ \ \text{and} \ \ g = t^{21} + \beta_1 t^{26} 
$$
with $\alpha_1, \alpha_2, \alpha_3, \beta_1 \in k$. Hence $$\xi = \frac{g}{f} = t^9 -\alpha_1t^{10} + \eta$$where $\eta \in V$ with $\rmo(\eta) \ge 11$. Because $t^{13}, t^4 \xi \in B$, we have $-\alpha_1t^{14} + t^4 \eta \in B$, which implies $\alpha_1 =0$, since $14 \not\in H_1=\left<4,9\right>$. Therefore, continuing the division algorithm, we get
$$
\xi = t^9 + (\beta_1 - \alpha_2)t^{14}  - \alpha_2(\beta_1 - \alpha_2)t^{19}  - \alpha_3t^{23} + \rho
$$ 
where $\rho \in V$ with $\rmo(\rho) \ge 24$. Consequently, expanding 
$$
\frac{g^2}{f} = g\xi = (t^{21} + \beta_1 t^{26}){\cdot} \left(t^9 + (\beta_1 - \alpha_2)t^{14}  - \alpha_2(\beta_1 - \alpha_2)t^{19}  - \alpha_3t^{23} + \rho\right)
$$
and considering the coefficient of $t^{35}$, we get $2\beta_1 = \alpha_2$, because $\frac{g^2}{f} \in I \subseteq A$ and $t^{35} \notin A$. Hence $$I =(t^{12} + 2\beta_1 t^{17} + \alpha_3 t^{26}, t^{21} +\beta_1t^{26}).$$ To see that $\beta_1 \ne 0$, we choose $\varphi, \psi \in A$ so that $t^{30} = f \varphi + g \psi$. Let $\varphi = \sum_{i=0}^{\infty}\varphi_i t^i$ and $\psi = \sum_{i=0}^{\infty}\psi_i t^i$ with $\varphi_i, \psi_i \in k$. Then, comparing the coefficients of $t^{30}$ and $t^{16}$ of both sides in the equation 
$$t^{30} = f  {\cdot} \sum_{i=0}^{\infty}\varphi_i t^i + g{\cdot}\sum_{i=0}^{\infty}\psi_i t^i,$$ 
we obtain 
$$
2 \beta_1 \varphi_{13}  + \varphi_4\alpha_3 + \beta_1 \psi_4 = 1 \ \ \text{and} \ \ \varphi_4 = 0. 
$$
Hence, $\beta_1 \ne 0$, which proves Assertion (1).

(2)~We may assume
$$
f = t^{16} + \alpha_1 t^{17} + \alpha_2 t^{21} + \alpha_3t^{26} \ \ \text{and} \ \ g = t^{25} + \beta_1 t^{26} 
$$
with $\alpha_1, \alpha_2, \alpha_3, \beta_1 \in k$. Since 
$$
\xi = \frac{g}{f} = t^9 + (\beta_1-\alpha_1)t^{10} - \alpha_1(\beta_1-\alpha_1)t^{11} + \rho
$$
with $\rmo(\rho) \ge 12$,
we get
$
\frac{g^2}{f} = g\xi \equiv t^{34} + (2 \beta_1 - \alpha_1)t^{35} \   \mod \ t^{36} V,
$
which implies $2 \beta_1 = \alpha_1$, because $\frac{g^2}{f} \in I$ but $t^{35} \not\in A$. We take $\varphi=\sum_{i=0}^{\infty}\varphi_i t^i, \psi=\sum_{i=0}^{\infty}\psi_i t^i\in A$ so that 
$$t^{30} = f  {\cdot} \sum_{i=0}^{\infty}\varphi_i t^i + g{\cdot}\sum_{i=0}^{\infty}\psi_i t^i.$$ Then, comparing the coefficients of $t^{30}$ and $t^{20}$ of both sides, we get $$2 \beta_1 \varphi_{13}  + \varphi_4\alpha_3 + \beta_1 \psi_4 = 1\ \  \text{and} \ \ \varphi_4 = 0.$$ Hence, $\beta_1 \ne 0$, which proves Assertion (2).
\end{proof}

If $H_1 =\left<4, 2n, 13\right>$, we have the following.

\begin{thm}\label{413}
Suppose that $H_1 =\left<4, 2n, 13\right>$ for with $n \in \{1, 3, 5, 7, 9, 11\}$. 
\begin{enumerate}[$(1)$]
\item If $n = 11$, then  
$$
I = (t^{4} + \alpha t^{13}, t^{26}) 
$$
where $\alpha \in k$.
\item If $n = 9$, then  
$$
I= (t^{8} + \alpha_1 t^{13} + \alpha_2t^{17}, t^{26}) 
$$
where $\alpha_1, \alpha_2 \in k$.
\item If $n = 7$, then  
$$
I= (t^{12} + \alpha_1 t^{13} + \alpha_2t^{17} + \alpha_3t^{21}, t^{26}) 
$$
where $\alpha_1, \alpha_2, \alpha_3 \in k$.
\item If $n = 5$, then
$$
I=(t^{16} + \alpha_1 t^{17} + \alpha_2t^{21} + \alpha_3t^{25}, t^{26}) 
$$
where $\alpha_1, \alpha_2, \alpha_3 \in k$. 
\item If $n = 3$, then 
$$
I=(t^{20} + \alpha_1 t^{21} + \alpha_2t^{25} + \alpha_3t^{29}, t^{26} + \beta t^{29}) 
$$
where $\alpha_1, \alpha_2, \alpha_3, \beta \in k$ and $\alpha_1^3 = 2\beta$. 
\item If $n=1$, then 
$$
I=(t^{24} + \alpha_1 t^{25} + \alpha_2t^{29} + \alpha_3t^{33}, t^{26} + \beta_1t^{29} + \beta_2 t^{33})  
$$
where $\alpha_1, \alpha_2, \alpha_3, \beta_1, \beta_2 \in k$ and
$$
\begin{cases}
\ \alpha_1 = 0 &\  \text{if} \  \ \ch k = 2, \\
\ \alpha_1 = \alpha_2 = \beta_1 = \beta_2 = 0 & \ \text{if} \ \ \ch k \ne  2.
\end{cases}
$$
\end{enumerate}
\end{thm}

\begin{proof}
For the same reason as in the proof of Corollary \ref{3.7}, we may assume
$$
f = t^{26-2n} + \alpha_1 t^{13} + \alpha_2t^{17} + \alpha_3 t^{21} + \alpha_4 t^{25} + \alpha_6t^{33} \ \ \text{and} \ \ g = t^{26} + \beta_1t^{29} + \beta_2t^{33}
$$
where $\alpha_1, \alpha_2, \alpha_3, \alpha_4, \alpha_5, \alpha_6, \beta_1, \beta_2 \in k$.  Notice that $t^{q} \in I$ for all $q \in H$ with $q \ge 36-2n=10+a$.

$(1)$~Therefore, because $a=4$, we have $t^{17} \in I$, whence $$I = (f, g, t^{17}) = (t^4 + \alpha_1t^{13}, t^{26}, t^{17}) = (t^4 + \alpha_1t^{13}, t^{26}),$$ where the last equality follows from the equation $t^{17} = (t^{4} + \alpha_1t^{13})t^{13} - \alpha_1t^{26}$. 

$(2)$~We have $t^{21}\in I$ and $$t^{21} = (t^8 + \alpha_1t^{13}+\alpha_2t^{17}) - (\alpha_1 + \alpha_2t^{4})t^{26},$$  whence $$I = (f, g, t^{21}) = (t^8 + \alpha_1t^{13}+\alpha_2t^{17}, t^{26}).$$ 

$(3)$~We have $t^{25} \in I$ and $$t^{25} = (t^{12} + \alpha_1t^{13} + \alpha_2t^{17} + \alpha_3t^{21})t^{13} - (\alpha_1 + \alpha_2t^4 + \alpha_3t^{8})t^{26}, $$ whence $$I = (f, g, t^{25}) = (t^{12} + \alpha_1 t^{13} + \alpha_2t^{17} + \alpha_3t^{21}, t^{26}).$$ 

$(4)$~We have $t^{29} \in I$ and $$t^{29} = (t^{16}+ \alpha_1t^{17} + \alpha_2t^{21} + \alpha_3t^{25}) t^{13}- (\alpha_1t^{4} + \alpha_2t^8 + \alpha_3t^{12})t^{26},$$ whence $$I = (f, g, t^{29}) = I=(t^{16} + \alpha_1 t^{17} + \alpha_2t^{21} + \alpha_3t^{25}, t^{26}).$$

$(5)$~We have $t^{33} \in I$. Hence $$I = (f, g, t^{33}) = (t^{20} + \alpha_1 t^{21} + \alpha_2t^{25} + \alpha_3t^{29}, t^{26} + \beta_1t^{29}, t^{33}).$$  Let $J = (t^{20} + \alpha_1 t^{21} + \alpha_2t^{25} + \alpha_3t^{29}, t^{26} + \beta_1t^{29}) \subseteq I$ and consider $v(J) = \{\rmo(h) \mid 0 \ne h \in J\}$. Then, $20, 26, 33 \in v(J)$, so that $36, 37, 38, 39 \in v(J)$ whence $J \supseteq t^{36} V=\fkc$. Because
$$
t^{34} = (t^{26} + \beta_1 t^{29})t^8 - \beta_1t^{37} \ \ \text{and} \ \ t^{33} = (t^{20} + \alpha_1t^{21} + \alpha_2t^{25} + \alpha_3t^{29})t^{13} - (\alpha_1t^{34} + \alpha_2t^{38} + \alpha_3t^{42}),
$$
we obtain $t^{34} \in J$ by the first equality, and hence $t^{33} \in J$ by the second one. Therefore, $I = J$, since $J \subseteq I=J +(t^{33})$. We now have
$$
\xi = \frac{g}{f} = t^6 -\alpha_1t^7 + \alpha_1^2t^8 + (\beta_1 - \alpha_1^3)t^9 - \alpha_1 (\beta_1 - \alpha_1^3)t^{10} + \rho$$
with $\rho \in V$ such that $\rmo(\rho) \ge 11$.
Consequently,  $2\beta_1 = \alpha_1^3$, because $$\frac{g^2}{f} = g \xi \equiv t^{32} -\alpha_1t^{33} + \alpha_1^2t^{34} + (\beta_1 - \alpha_1^3)t^{35}  \ \mod \ t^{36}V$$ and because $\frac{g^2}{f} \in A$ but $t^{35} \not\in A$.

$(6)$~Since $f = t^{24} + \alpha_1t^{25} + \alpha_2t^{29} + \alpha_3t^{33}$ and $g = t^{26} + \beta_1 t^{29} + \beta_2 t^{33}$, we get
\begin{eqnarray*}
\xi = \frac{g}{f} &=& t^2 - \alpha_1t^{3} + \alpha_1^2t^4 + (\beta_1 - \alpha_1^3)t^5 -\alpha_1(\beta_1-\alpha_1^3)t^6 + (\alpha_1^2 (\beta_1 - \alpha_1^3)-\alpha_2)t^7  \\ 
&+&  (2\alpha_1\alpha_2 - \beta_1\alpha_1^3 +\alpha_1^6)t^8 +  (\alpha_2 - \alpha_1^2\alpha_2 -\alpha_1(2\alpha_1\alpha_2 - \beta_1\alpha_1^3 +\alpha_1^6))t^9 + \rho
\end{eqnarray*}
with $\rho \in V$ with $\rmo(\rho) \ge 10$. Therefore
\begin{eqnarray*}
\frac{g^2}{f} = g\xi &=& t^{28} -\alpha_1t^{29} + \alpha_1^2t^{30} +(2\beta_1 - \alpha_1^3) t^{31} + (-(\beta_1 - \alpha_1^3)\alpha_1 - \beta_1 \alpha_1)t^{32} \\
&+& (2\beta_1\alpha_1^2 - \alpha_1^5 -\alpha_2)t^{33} + ((2\alpha_1\alpha_2 - \beta_1\alpha_1^3 +\alpha_1^6) + \beta_1(\beta_1- \alpha_1^3))t^{34} \\
&+& (2\beta_2 -3\alpha_1^2\alpha_2+2\beta_1\alpha_1^4 - \alpha_1^7 - \beta_1^2\alpha_1)t^{35} + \eta,
\end{eqnarray*}
where $\rmo(\eta) \ge 36$. Consequently
$$2\beta_1 = \alpha_1^3\ \ \text{and}\ \  2\beta_2 = 3 \alpha_1^2 \alpha_2 + \beta_1^2 \alpha_1,$$
since $\frac{g^2}{f}=g\xi \in A$. Therefore, if $\ch k = 2$, then $\alpha_1^3 = 0$, so that $\alpha_1 = 0$. 

Suppose that $\ch k \ne 2$. Then, since $$\frac{g^2}{f} = g\xi \equiv t^{28} -\alpha_1t^{29} + \alpha_1^2t^{30}  -\alpha_2t^{33} + (2\alpha_1\alpha_2  + \beta_1^2)t^{34} \ \mod \ t^{36}V$$ and since $t^{34} \in I$ (remember $c-(b-a)=10+a=34$), we have $$t^{28} -\alpha_1t^{29} + \alpha_1^2t^{30}  -\alpha_2t^{33} \in I.$$ 
Notice that $$t^4 f = t^{28} + \alpha_1t^{29} + \alpha_2t^{33} + \alpha_3t^{37}$$ and that $\fkc = t^{36}V \subseteq I$. We then have $2t^{28} + \alpha_1^2t^{30} \in I$ and, writing $$2t^{28} + \alpha_1^2t^{30} = f\varphi + g \psi$$ with $\varphi = \sum_{i=0}^{\infty}\varphi_i t^i, \psi = \sum_{i=0}^{\infty} \psi_i t^i \in A$  ($\varphi_i, \psi_i \in k$) and comparing the coefficients of $t^{28}, t^{29}, t^{30}$, and $t^{33}$ in both sides of the above equation, we have  
$$
\varphi_4 = 2, \ \ \varphi_4 \alpha_1=0, \ \ \psi_4 = \alpha_1^2, \ \ \text{and} \ \ \varphi_8 \alpha_1 + \varphi_4\alpha_2 + \beta_1 \psi_4=0.
$$
Therefore, $\alpha_1 = 0$. Because $2\beta_1 = \alpha_1^3$ and $2\beta_2 = 3 \alpha_1^2 \alpha_2 + \beta_1^2 \alpha_1$, we get $\beta_1 = \beta_2 =0$, whence $\alpha_2=0$. This completes the proof of Theorem \ref{413}. 
\end{proof}

We are ready to prove Theorem \ref{e=4}.

\begin{proof}[{Proof of Theorem {\rm \ref{e=4}}}] 
It suffices to show that the listed ideals are all Ulrich. We will check it, following the grouping of cases given by Theorems \ref{b-a=9} and \ref{413}. Because the proof of Cases (1) and (2) of Theorem \ref{b-a=9} and the proof of Cases (1), (2), (3), and (4) (resp. (5) and (6)) of Theorem \ref{413} are almost the same, we shall consider Case (1) of Theorem \ref{b-a=9} and Cases (1) and (5) of Theorem \ref{413} only.

Case (1) of Theorem \ref{b-a=9}.~Let $f = t^{12}+2\beta t^{17} + \alpha t^{26}$, $g= t^{21} + \beta t^{26}$ where $\alpha, \beta \in k$ such that $\beta \ne 0$ and set $I = (f,g)$. We want to show $I \in \calX_A$. Let $L = v(I)$. Then, $28,32,36 \in L$ since $12 \in L$, while $29,33,34,37,38 \in L$ since $21 \in L$. We have $\beta t^{34} + \beta^2 t^{39} = \beta t^{13}g \in I$, while $\beta t^{34} + \alpha t^{43} = t^{17}f - t^8 g \in I$. Hence, $\beta^2 t^{39} - \alpha t^{43} \in I$, so that $39 \in L$ (since $\beta \ne 0$). Thus, $\fkc = t^{36}V  \subseteq I$, because $36, 37, 38, 39 \in L$. Therefore, we have  $t^{33} \in I$ (resp. $t^{32} \in I$), considering $t^{12}g$ (resp. $t^{20}f$). Because $\beta t^{30} + \alpha t^{39} = t^{13}f - t^4 g \in I$, we get $t^{30} \in I$. Hence, $t^{30}, t^{32}, t^{33}, t^{39} \in I$. Thus, $t^q \in I$ for all $q \in H$ such that $q \ge 30$. Remember that $$\xi =\frac{g}{f} = t^9 - \beta t^{14} +2\beta^2 t^{19} + \rho$$ where $\rho \in V$ with $\rmo(\rho) \ge 22$, and we get
$$\frac{g^2}{f} = g\xi \equiv t^{30} + \beta^2 t^{40}~\mod~t^{43}V.$$ Hence, $\frac{g^2}{f} \in I$, because $t^{30} \in I$ and $t^{36}V \subseteq I$. Thus, $I^2 = fI$. On the other hand, since the $k$-space $A/I$ is spanned by the images of $$1, \ t^4, \ t^8, \ t^{13},\  t^{17},\  t^{26}$$ we get
$\ell_A(A/I) \le 6$, so that 
the epimorphism $$\varphi : A/I \to I/(f),\ \ \varphi(1~\mod~I)=g~\mod~(f)$$
of $A$-modules is an isomorphism, because $\ell_A(I/(f)) \ge 6$ (since $\ell_A(A/I)=12$). Hence, $A/I \cong I/(f)$, so that  $I$ is an Ulrich ideal of $A$ with $\ell_A(A/I)=6$, and the images of the above six monomials form a $k$-basis of $A/I$.

Case (1) of Theorem \ref{413}.~Let $f = t^{4} + {\alpha}t^{13}$ ($\alpha \in k$) and $g= t^{26}$, and set $I = (f, g)$. We want to show $I \in \calX_A$. Firstly, notice that $t^{17} \in I$, since
$$t^{17} = (t^{4}+\alpha t^{13})t^{13} -\alpha t^{26}.$$
Therefore, $t^{37},t^{38},t^{39} \in I$, while $t^{36} \in I$, since $t^{32}f = t^{36} + \alpha t^{45}$ and $t^{45} \in I$. Hence, $\fkc =t^{36}V \subseteq I$. Because $\rmo(\frac{g^2}{f})=48$, we get $g^2 \in f\fkc \subseteq fI$, whence $I^2=fI$. On the other hand, because the $k$-space $A/I$ is spanned by the images of $\{t^q \mid q \in H, q \le 34\}$ and because among them, there are relations induced from the vanishing
$$f = t^{4} + {\alpha}t^{13} \equiv 0~\mod ~I\ \ \text{and}\ \ g=t^{26} \equiv 0 ~\mod~I$$
of $f$ and $g$, the $k$-space $A/I$ is spanned by the images of $1, t^{13}$, whence $\ell_A(A/I) \le 2$. Therefore, the epimorphism
$$\varphi : A/I \to I/(f),\ \ \varphi(1~\mod~I)=g~\mod~(f)$$
of $A$-modules is an isomorphism, because $\ell_A(I/(f) \ge 2$ (since $\ell_A(A/(f)=4$). Thus, $I/(f) \cong A/I$ and $\ell_A(A/I)=2$. Hence, $A/I$ possesses the images of $1, t^{13}$ as a $k$-basis, and  $I$ is an Ulrich ideal of $A$.

Case (5) of Theorem \ref{413}.~Let $$f = t^{20} + \alpha_1 t^{21} + \alpha_2t^{25} + \alpha_3t^{29} \ \ \text{and}\ \ g= t^{26} + \beta t^{29}$$
where $\alpha_1, \alpha_2, \alpha_3, \beta \in k$  and $\alpha_1^3 = 2\beta$. We set $I = (f, g)$ and $L = v(I)$. Then, $20, 26, 33, 39 \in L$. Hence, $36,37,38,39 \in L$, so that $\fkc = t^{36}V \subseteq I$. On the other hand, because 
$$t^{34}= t^8(t^{26}+ \beta t^{29}) -\beta t^{37} \ \ \text{and}\ \ t^{33} = t^{13}\left(t^{20}+\alpha_1 t^{21}+\alpha_2 t^{25} + \alpha_3 t^{29}\right)-(\alpha_1t^{34} + \alpha_2 t^{38} + \alpha_3 t^{42}),$$
we get $t^{34}, t^{33} \in I$. We furthermore have $t^{30}, t^{32}, t^{39} \in I$, because $t^{4}g, t^{12}f, t^{13}g \in I$. Hence, $t^q \in I$ for all $q \in H$ such that $q \ge 30$. Therefore, because $\alpha_1^3 = 2\beta$, we have
$$
\frac{g^2}{f} = g\xi \equiv t^{32} - \alpha_1t^{33} + \alpha_1^2t^{34} \ \mod \ t^{36}V,
$$
which yields that $\frac{g^2}{f} \in I$, whence $I^2 = fI$. Because the $k$-space $A/I$ is spanned by the images of $\{t^q \mid q \in H, q < 30\}$ and because there are relations among them induced by the vanishing
$$f = t^{20} + \alpha_1 t^{21} + \alpha_2t^{25} + \alpha_3t^{29} \equiv 0~\mod ~I\ \ \text{and}\ \ g= t^{26} + \beta t^{29} \equiv 0 ~\mod~I,$$
$A/I$ is spanned by the images of the following ten monomials
$$
1, t^4,t^8,t^{12}, t^{13},t^{16}, t^{17}, t^{21}, t^{25}, t^{29}.
$$
Therefore, $\ell_A(A/I) \le 10$, so that the epimorphism
$$\varphi : A/I \to I/(f),\ \ \varphi(1~\mod~I)=g~\mod~(f)$$
is an isomorphism, because $\ell_A(I/(f) \ge 10$ (since $\ell_A(A/(f)=20$). Thus, $I/(f)$ is a free $A/I$-module, and $A/I$ possesses the images of the above ten monomials as a $k$-basis. In particular, $I$ is an Ulrich ideal of $A$.

Let us check the second assertion of Theorem \ref{e=4}. For example, we consider Case (1) of Theorem \ref{b-a=9}. Let 
\begin{eqnarray*}
f&=&t^{12} + 2\beta  t^{17} + \alpha  t^{26},\\ 
g&=& t^{21} + \beta t^{26},\\
f_1&=&t^{12} + 2\beta_1  t^{17} + \alpha_1 t^{26},\\ 
g_1&=& t^{21} + \beta_1 t^{26}\\
\end{eqnarray*}
where $\alpha, \beta, \alpha_1, \beta_1 \in k$ such that $\beta \ne 0, \beta_1 \ne 0$, and assume that $I =(f,g)=(f_1,g_1)$. Then, since $f - f_1, g-g_1 \in I$ and the images of $1, t^4, t^8, t^{13}, t^{17}, t^{26}$ form a $k$-basis of $A/I$ (as we have shown above), we readily get $\alpha =\alpha_1$ and $\beta = \beta_1$. This argument works also for the proof of the other cases. This completes the proof of Theorem \ref{e=4}.
\end{proof}

 
\begin{remark}
We are able to pinpoint the set $\calX_{k[[t^5, t^{11}]]}$ also. The results on $k[[t^5, t^{11}]]$ are more subtle and the proof is more formidable than those of the case of $k[[t^4,t^{13}]]$, which we shall give on another occasion.
\end{remark}


\section{The case of three-generated numerical semigroup rings}
In this section we explore the semigroup ring of a three-generated numerical semigroup.

Throughout, let $a, b, c \in \Bbb Z$ be positive integers with $\GCD(a, b, c) = 1$. We set $H = \left<a, b, c\right>$ and assume that $H$ is minimally generated by $a,b,c$. Let $k[[t]]$ be the formal power series ring over a field $k$ and $k[[H]] = k[[t^a, t^b, t^c]]$. We set $\m = (t^a, t^b, t^c)$, the maximal ideal of $A=k[[H]]$. We are interested in the size of the set $\calX_A$.

We begin with the case where $H$ is not symmetric. For a given finitely generated $A$-module $M$ let $P_M^A(t)$ stand for the Poincar\'e series 
$$
P_M^A(t) = \sum_{n=0}^{\infty}\beta_n^A(M)t^n \in \Bbb Z[[t]]
$$
of $M$ over $A$, where $\beta_n^A(M)$ denotes the $n$-th Betti number of $M$. With this notation we have the following.

\begin{thm}\label{6.1}
Suppose that $A=k[[H]]$ is not a Gorenstein ring. Then the Betti numbers of the residue class field $A/\m$ of $A$ are given by
$$
\beta_n^A(A/\m)=
\begin{cases}
1 & (n=0)\\
3 \cdot 2^{n-1} & (n>0).
\end{cases}
$$
Hence
$$
P_{A/\m}^A(t) = \frac{1+t}{1-2t}. 
$$
\end{thm}

\begin{proof}
Let $k[[X,Y,Z]]$ be the formal power series ring over $k$ and let $\varphi: k[[X,Y,Z]] \to A$ be the $k$-algebra map defined by $\varphi(X)=t^a, \varphi(Y)=t^b$, and $\varphi(Z) = t^c$. Then, since $A$ is not a Gorenstein ring, thanks to \cite[Section III]{Herzog}, we have an isomorphism 
$$
A \cong k[[X, Y, Z]]/{\rmI}_2
\left(\begin{smallmatrix}
X^{\alpha} & Y^{\beta} & Z^{\gamma} \\
Y^{\beta'} &  Z^{\gamma'} & X^{\alpha'} 
\end{smallmatrix}\right)
$$
of $k$-algebras for some positive integers $\alpha, \beta, \gamma, \alpha', \beta', \gamma'$, where ${\rmI}_2\left(\begin{smallmatrix}
X^{\alpha} & Y^{\beta} & Z^{\gamma} \\
Y^{\beta'} &  Z^{\gamma'} & X^{\alpha'} 
\end{smallmatrix}\right)$ denotes the ideal generated by the $2\times 2$-minors of the matrix $\left(\begin{smallmatrix}
X^{\alpha} & Y^{\beta} & Z^{\gamma} \\
Y^{\beta'} &  Z^{\gamma'} & X^{\alpha'} 
\end{smallmatrix}\right)$. Hence 
$$
A/(t^a) \cong k[Y, Z]/{\rmI}_2
\left(\begin{smallmatrix}
0 & Y^{\beta} & Z^{\gamma} \\
Y^{\beta'} &  Z^{\gamma'} & 0
\end{smallmatrix}\right)
= k[Y, Z]/(Y^{\beta + \beta'}, Y^{\beta'}Z^{\gamma}, Z^{\gamma+\gamma'})
$$
as a $k$-algebra. Let
$$
B = k[Y, Z]/(Y^{\beta + \beta'}, Y^{\beta'}Z^{\gamma}, Z^{\gamma+\gamma'})
$$
and let $y, z$ denote the images of $Y, Z$ in $B$, respectively. Then, because
$$P_{B/(y, z)}^B(t) = \displaystyle\frac{P^A_{A/\m}(t)}{1 + t},$$
we readily get $P^A_{A/\m}(t) = \frac{1+t}{1-2t}$,
once we have $$P_{B/(y, z)}^B(t) =\frac{1}{1-2t}=1 + 2t + 4t^2 + \cdots + 2^nt^n + \cdots.$$ To see it, we consider the minimal $B$-free resolution of $B/(y, z)$. It is straightforward to check that
$$
B^{\oplus 4} \overset{\Bbb M_1}{\longrightarrow} B^{\oplus 2} \overset{\Bbb M_0}{\longrightarrow} B \overset{\varepsilon}{\longrightarrow} B/(y, z) \longrightarrow 0  
$$
forms a part of minimal $B$-free resolution of $B/(y, z)$, where $\varepsilon$ is the canonical epimorphism, $${\Bbb M}_0 = (y \ z), \ \ \text{and} \ \ {\Bbb M}_1 = \left(\begin{smallmatrix}
y^{\beta + \beta' -1} & y^{\beta'-1}z^{\gamma} & 0 & z \\
0 & 0 & z^{\gamma + \gamma'-1} & -y 
\end{smallmatrix}\right).$$  We want to know the remainder part of the resolution. To do this, firstly let $\alpha \in \Ker {\Bbb M_1}$ and write $\alpha = {}^t(\alpha_1, \alpha_2, \alpha_3, \alpha_4)$ with $\alpha_i \in B$, where ${}^t(*)$ denotes the transpose. Then, since $$\alpha_3 z^{\gamma + \gamma' -1} - \alpha_4 y=0 \ \ \text{in}\ \ B,$$ we have
$$
\alpha_3= yf_4 + zf_3 \ \ \text{and}\ \ \alpha_4=z^{\gamma + \gamma'-1}f_4 - y^{\beta + \beta'-1}f_1 -y^{\beta'-1}z^{\gamma}f_2
$$
for some $f_1, f_2, f_3, f_4 \in B$, 
so that 
$$
\alpha_1 y^{\beta + \beta'-1}  + \alpha_2y^{\beta'-1}z^{\gamma} + (z^{\gamma + \gamma'-1}f_4 - y^{\beta + \beta'-1}f_1 -y^{\beta'-1}z^{\gamma}f_2)z = 0.
$$
Hence
$$
\alpha_1-f_1z - g_1 y = g_4z^{\gamma} \ \ \text{and} \ \ g_2y + g_3 z^{\gamma'}-\alpha_2 + g_2z = g_4y^{\beta}
$$
for some $g_1, g_2, g_3, g_4 \in B$, which shows that $\alpha = {}^t(\alpha_1, \alpha_2, \alpha_3, \alpha_4)$ is contained in the $B$-submodule of $B^4$ generated by the columns of the matrix
$$
{\Bbb M}_2 =
\left(\begin{smallmatrix}
y & z & 0 & 0 & 0 & 0 & 0 & 0\\
0 & 0 & y & z & 0 & 0 & 0 & 0\\
0 & 0 & 0 & 0 & y & z & 0 & 0\\
0 & -y^{\beta+\beta'-1} & 0 & -y^{\beta'-1}z^{\gamma} & z^{\gamma+ \gamma' -1} & 0 & y^{\beta+\beta'-1}z^{\gamma-1} & y^{\beta'-1}z^{\gamma+ \gamma'-1} 
\end{smallmatrix}\right).
$$
Hence ${\Ker \Bbb M}_1 = \Im {\Bbb M}_2$. Next, let $\alpha \in {\Ker \Bbb M}_2$ and write $\alpha = {}^t(\alpha_1, \alpha_2, \ldots, \alpha_8)$ with $\alpha_i \in B$. Then, because $\alpha_1 y + \alpha_2z=0$, $\alpha_3 y + \alpha_4z=0$, and $\alpha_5 y + \alpha_6z=0$, we get
$$
\left(\begin{smallmatrix}
\alpha_1 \\
\alpha_2
\end{smallmatrix}\right), 
\left(\begin{smallmatrix}
\alpha_3 \\
\alpha_4
\end{smallmatrix}\right), 
\left(\begin{smallmatrix}
\alpha_5 \\
\alpha_6
\end{smallmatrix}\right)\in \Im {\Bbb M}_1
$$
whence $\alpha_2, \alpha_4, \alpha_6 \in (z^{\gamma+ \gamma'-1}, y)$. 
This implies $\alpha_2 y^{\beta+ \beta'-1} + \alpha_4y^{\beta'-1}z^{\gamma} - \alpha_5z^{\gamma+ \gamma' -1}=0$ in $B$, so that 
$\alpha_7 y^{\beta+ \beta'-1}z^{\gamma-1} + \alpha_8y^{\beta'-1}z^{\gamma + \gamma' -1}=0$, whence 
$$
\alpha_7 - f_1y - f_3z = z^{\gamma'}f_5 \ \ \text{and} \ \ \alpha_8 - f_4 y - f_2z = -y^{\beta}f_5
$$
for some $f_1, f_2, f_3, f_4, f_5 \in B$. Thus 
$$
\left(\begin{smallmatrix}
\alpha_7 \\
\alpha_8
\end{smallmatrix}\right)
\in \left<
\left(\begin{smallmatrix}
y \\
0
\end{smallmatrix}\right), 
\left(\begin{smallmatrix}
z \\
0
\end{smallmatrix}\right), 
\left(\begin{smallmatrix}
0 \\
y
\end{smallmatrix}\right), 
\left(\begin{smallmatrix}
0 \\
z
\end{smallmatrix}\right)
\right>,
$$
which guarantees that $\alpha$ is contained in the $B$-submodule of $B^{14}$ generated by the columns of the matrix
$$
{\Bbb M}_3
=
\left(\begin{smallmatrix}
{\Bbb M}_1 &  & && \\
& {\Bbb M}_1 & && \\
& &  {\Bbb M}_1 & & \\
& & &  {\Bbb M}_0   &\\ 
& & &  & {\Bbb M}_0   \\ 
\end{smallmatrix}\right).
$$
Consequently, the sequence 
$$
B^{\oplus 16} \overset{{\Bbb M}_3}{\longrightarrow} B^{\oplus 8} \overset{{\Bbb M}_2}{\longrightarrow} B^{\oplus 4} \overset{{\Bbb M}_1}{\longrightarrow} B^{\oplus 2} \overset{{\Bbb M}_0}{\longrightarrow} B 
 \overset{\varepsilon}{\longrightarrow}
B/(y, z) \longrightarrow 0 
$$
forms a part of the minimal $B$-free resolution of $B/(y, z)$. Since the matrix ${\Bbb M}_3$ consists of submatrices ${\Bbb M}_0$ and ${\Bbb M}_1$, the Poincar\'e series of $B/(y, z)$ has the form
$$
P_{B/(y, z)}^B(t) = 1 + 2t + 4t^2 + \cdots + 2^nt^n + \cdots
$$
as claimed. 
\end{proof}

For a moment let $A$ be an arbitrary Noetherian local ring $A$ with the maximal ideal $\m$, the Poincar\'e series $P^A_{A/\m}(t)$ of $A/\m$ over $A$ is not necessarily a rational function (\cite{A}). On the other hand,  J.-P. Serre proved that $P^A_{A/\m}(t)$ is coefficientwise bounded above by the rational series
$$
\frac{(1+t)^n}{1- \sum_{i \ge 1} \ell_A(\rmH_i(x_1, x_2, \ldots, x_n; A))t^{i+1}},
$$
where $x_1, x_2, \ldots, x_n$ ($n = \mu_A(\m)$) is a minimal system of generators of $\m$ and $\rmH_i(x_1, x_2, \ldots, x_n; A)$ denotes the $i$-th Koszul homology module of $A$ with respect to the sequence $x_1, x_2, \ldots, x_n$. With this notation, $A$ is called a {\it Golod} ring, if $P^A_{A/\m}(t)$ coincides with the upper bound given by Serre.

The following result is known (see, e.g., \cite{GPW}) and it is equivalent to the assertion of Theorem \ref{6.1}. Here let us give it as a direct consequence of the proposition.

\begin{cor}[{cf., e.g., \cite{GPW}}]\label{4.2}
Every three-generated non-Gorenstein numerical semigroup ring $k[[t^a,t^b,t^c]]$ is Golod. 
\end{cor}

\begin{proof}
Let $S=k[[X, Y, Z]]$ be the formal power series ring over $k$. Then, the $S$-module $A$ has a minimal free resolution
$$ 0 \to S^2 \overset{\left(\begin{smallmatrix}
X^{\alpha}&Y^{\beta'}\\
Y^{\beta}&Z^{\gamma'}\\
Z^{\gamma}&X^{\alpha'}\\
\end{smallmatrix}\right)}
{\longrightarrow}S^3 \to S \to A \to 0,$$
whence Theorem \ref{6.1} tells us that
$$
P_{A/\m}^A(t) = \frac{1+t}{1-2t} = \frac{(1+t)^3}{1-3t^2-2t^3} = \frac{P_{S/\n}^S(t)}{1- t\cdot (P_{A}^S(t)-1)},
$$ 
where $\n=(X, Y, Z)$ denotes the maximal ideal of $S$. Therefore, the natural surjection  $S \to A$ is a Golod homomorphism, so that $A$ is a Golod ring (\cite[Remark, page 32]{Avramov}).
\end{proof}

We say that a Noetherian local ring $A$ is {\it $G$-regular}, if every totally reflexive $A$-module is free (\cite{greg}), and \cite[Example 3.5]{AM} guarantees that every Golod local ring which is not a hypersurface must be $G$-regular. Consequently, we readily get the following.

\begin{cor}\label{4.3}
Every three-generated non-Gorenstein numerical semigroup ring $k[[t^a, t^b, t^c]]$ contains no Ulrich ideals generated by two elements.
\end{cor}

When $H=\left<a,b,c\right>$ is symmetric, few things are known about the size of $\calX_{k[[H]]}$. Closing this paper, let us note a part of them.

First of all, remember that $H$ is symmetric if and only if $k[[H]]$ is a complete intersection (\cite[Theorem 3.10]{Herzog}, \cite[Corollary 10.5]{RG}). If $H$ is symmetric, then as is partially stated in the proof of Proposition \ref{5.1} (2), $H$ is obtained by a gluing of a two-generated numerical semigroup $H'$ and $\Bbb N$ (\cite[Section 3]{Herzog}, \cite[Proposition 3]{W}). Let us explain more precisely, preparing new notation.

Let $\alpha, \beta \in \Bbb Z$ be positive integers such that $\GCD(\alpha, \beta) = 1$. We set $H' = \left<\alpha, \beta \right>$ and assume that $\mu(H') = 2$. Choose $a \in H'$ and $b \in \Bbb  N$ so that $a,b$ satisfy  the conditions that $a > 0$, $b > 1$, $a \not\in \{\alpha, \beta\}$, and $\GCD(a, b) = 1$.  Hence,  $\GCD(b\alpha, b\beta, a) = 1$. We consider the numerical semigroup $
H = \left<b\alpha, b\beta, a \right>,
$
and call it the {\it gluing} of $H'$ and $\Bbb N$ with respect to $a \in H'$ and $b \in \Bbb N$. Notice that $H = \left<b\alpha, b\beta, a \right>$ is actually symmetric. In fact, let $a = \alpha \ell + \beta m$ with integers $\ell, m \ge 0$. Let $k[[X, Y, Z]]$ be the formal power series and consider the $k$-algebra map 
$\varphi : k[[X, Y, Z]] \to V$
defined by $$\varphi(X) = t^{b \alpha}, \ \ \varphi(Y) = t^{b \beta}, \ \ \text{and}\ \ \varphi(Z) = t^{a}.$$ We then have the isomorphism 
$$
A \cong k[[X, Y, Z]]/(X^{\beta} - Y^{\alpha}, Z^{b} - X^{\ell}Y^m)
$$
of $k$-algebras. Every symmetric three-generated numerical semigroup is obtained by gluing some $H'$ and $\Bbb N$ with respect to suitable $a \in H'$ and $b \in \Bbb N$.

Assume that our symmetric semigroup $H$ has the form $H=\left<b \alpha, b\beta, a \right>$ stated above. We then have  the following. 
 
\begin{prop}\label{5.6}
Suppose that one of the following conditions is satisfied.
\begin{enumerate}[$(1)$]
\item $b$ is even and $\ell \ge 2$.
\item $b$ is even and $m \ge 2$.
\item either $\alpha$ or $\beta$  is even. 
\end{enumerate}
Then $A$ admits at least one Ulrich ideal.
\end{prop}
 
\begin{proof} We identify $A=k[[X, Y, Z]]/(X^{\beta} - Y^{\alpha}, Z^{b} - X^{\ell}Y^m)$ and let $x$, $y$, and $z$ respectively denote the images of $X$, $Y$, and $Z$ in $A$. We set $I_1 =(x, z^{\frac{b}{2}})$ for Case (1), $I_2 = (y, z^{\frac{b}{2}})$ for Case (2), $I_3= (x, y^{\frac{\alpha}{2}})$ for Case (3) if $\alpha$ is even, and $I = (x, y^{\frac{\beta}{2}})$ for Case (3) if $\beta$ is even. Then, these ideals are Ulrich ideals of $A$. Let us check Case (1) only. The others are similarly proved.

$(1)$~Since $z^{b} = x^{\ell}y^m$ and $\ell \ge 2$, $I^2 = xI + (z^{b)} = xI$. Notice that $$(X, X^{\beta} - Y^{\alpha}, Z^{b} - X^{\ell}Y^m) = (X, Y^{\alpha}, Z^{b}).$$ We then have $(x):_AI = I$, because $$(X, Y^{\alpha}, Z^{b}):_{k[[X,Y,Z]]}Z^{\frac{{b}}{2}} = (X, Y^{\alpha}, Z^{\frac{{b}}{2}}),$$ so that $I$ is a good ideal of $A$ in the sense of \cite{GIW}.  Therefore, $I \in \calX_A$ by \cite[Corollary 2.6 (b)]{GOTWY}, because $A/I \cong k[[Y]]/(Y^{\alpha})$ is a Gorenstein ring.
\end{proof}

For example, choose an even integer $\ell \ge 8$ so that $\operatorname{GCD}(3,\ell)=1$ and set $H'=\left<3, \ell\right>$. Let $a \in H'$ and $b \in \Bbb  N$ such that $a > 0$, $b > 1$, $a \not\in \{3, \ell\}$, and $\GCD(a, b) = 1$. Then, thanks to Theorem \ref{e=3} and Proposition \ref{5.6}, the semigroup ring $k[[t^{3b}, t^{b\ell}, t^{a}]]$ of $H=\left<3b, b\ell, a\right>$ admits the Ulrich ideal $I = (t^{3b}, t^{\frac{b\ell^2}{2}})$. We have $\calX_{k[[t^8,t^9,t^{15}]]} = \emptyset$ and the symmetric semigroup $H=\left<8,9,15 \right>$ does not satisfy any one of the conditions stated in Proposition \ref{5.6}.  On the other hand, as for $H = \left<8, 15, 25\right>$ or $H = \left<8, 21, 35\right>$, we are sure that the semigroup ring $A=k[[H]]$ contains no Ulrich ideals generated by monomials in $t$, although we don't know whether $\calX_A=\emptyset$ or not. Thus, at this moment, for a general three-generated symmetric semigroup $H$ the answer to the question of how large the set $\calX_{k[[H]]}$ is remains open rather far from comprehension.

\begin{ac}
{\rm The authors are grateful to S. Iai and N. Matsuoka for their helpful and simultaneous discussions during this research. The authors are grateful also to H. Matsui for his valuable advice, which has led the authors to Corollary \ref{4.2}.}
\end{ac}


\end{document}